\documentclass[11pt,leqno]{article} 
\usepackage[english]{babel}
\usepackage[utf8]{inputenc} 
\usepackage[pdftex]{graphicx}
\usepackage{pict2e}
\usepackage{cancel}
\usepackage{eurosym}
\usepackage{amsmath,amsfonts,amssymb,amsthm,epsfig,epstopdf,titling,url,array}
\usepackage{mathtools}
\usepackage[colorlinks=true]{hyperref}

\advance\textheight by \topskip
\newcommand{\Z}{\mathbb{Z}}
\newcommand{\N}{\mathbb{N}}

\newcommand{\R}{\mathbb{R}}

\newcommand{\del}{\partial}
\newcommand{\eps}{\varepsilon}

\newcommand{\eg}[0]{\emph{e.g.} }
\newcommand{\ie}[0]{\emph{i.e.} }

\theoremstyle{plain}
\newtheorem{thm}{Theorem}[section]
\newtheorem{lem}[thm]{Lemma}
\newtheorem{prop}[thm]{Proposition}
\newtheorem{cor}[thm]{Corollary}
\newtheorem{fac}[thm]{Fact}
\newtheorem*{thm*}{Theorem}

\theoremstyle{definition}
\newtheorem{defn}[thm]{Definition}

\newtheorem{remk}[thm]{Remark}
\newtheorem{ques}[thm]{Question}

\DeclareMathOperator{\Sep}{Sep}

\title{Separation profiles, isoperimetry, growth and compression}
\author{Corentin Le Coz \thanks{Supported by projects ANR-14-CE25-0004 GAMME and ANR-16-CE40-0022-01 AGIRA. } \and Antoine Gournay \thanks{Supported by the ERC Consolidator Grant No. 681207, ``Groups, Dynamics, and Approximation''.}}

\begin{document}

\maketitle

\begin{abstract}
We give lower and upper bounds for the separation profile (introduced by Benjamini, Schramm \& Timár) for various graphs using the isoperimetric profile, growth and Hilbertian compression. For graphs which have polynomial isoperimetry and growth, we show that the separation profile $\Sep(n)$ is also bounded by powers of $n$. For many amenable groups, we show a lower bound in $n/ \log(n)^a$ and, for any group which has a non-trivial compression exponent in an $L^p$-space, an upper bound in $n/ \log(n)^b$. We show that solvable groups of exponential growth cannot have a separation profile bounded above by a sublinear power function. In an appendix, we introduce the notion of local separation, with applications for percolation clusters of $ \Z^{d} $ and graphs which have polynomial isoperimetry and growth.
\end{abstract}

%____________________
\section{Introduction}\label{sintro}
%____________________

The separation profile was first defined by Benjamini, Schramm \& Timár in \cite{BST}. The separation function at $n$ is the supremum over all subgraphs of size $n$, of the number of vertices needed to be removed from the subgraph, in order to cut into connected pieces of size at most $n/2$. The introduction of this function was motivated by the study of regular maps between metric spaces, because the separation profile is monotonous under
regular maps (up to constant factors). A map between two graphs of uniformly bounded degrees is said to be regular if the two following conditions are satisfied: first, if it is Lipschitz at large scale, \emph{i.e.} satisfying  a Lipschitz inequality restricting to pair of points such that the distance between them is larger that some constant, second if the cardinality of the preimages of singletons is uniformly bounded (see Definition \ref{dregular}); for example, quasi-isometric and coarse embeddings between Cayley graphs are regular maps.

Hume's work about linear separation profiles (see \cite{H}) led to an equivalent definition of the separation profile, using the notion Cheeger constants $h$. 
The definition of the separation profile that we use is the following: (see section \ref{sdefi} for the details)

\begin{defn}[Separation profile]
$$ \Sep (n) = \sup_{F\subset V\Gamma, |F|\leq n} |F| \cdot h(F)
\quad\text{where $ h(F)$ denotes the Cheeger constant of $ F $}.
$$

\end{defn}

This was studied by Hume in \cite{H} with the aim of finding large classes of expanders. His work was continued by Hume, Mackay \& Tessera \cite{HMT} who introduced $\ell^p$-variants of these profiles and recently Hume \& Mackay \cite{HM} studied the case of groups with low separation profiles ($ \prec \log(n) $). On the opposite side, the separation profiles of expanders is linear (along a subsequence). \\[-1ex]

The subject matter of the current paper is to estimate the separation profile for various types of graphs, using other known information such as growth, isoperimetry and compression. For instance separation and isoperimetric profiles both have a functional definition. As remarked in \cite{HMT}, they roughly differ by boundary conditions: Neumann conditions for the separation profile and Dirichlet conditions for the isoperimetric profile. It is then natural to compare those two profiles. For this comparison, our central argument is contained in Lemma \ref{tbase-l}.

Our focus is mostly on Cayley graphs, because this is where most information is available on other properties of the graph (such as growth, isoperimetry and compression). However our methods do not really require the high level of regularity that Cayley graphs possess.
\\
\\
A first group of results regards graphs which are ``polynomial'' in some sense. 
Here $\Lambda$ denotes the isoperimetric profile. 
For example, if for some $K>0$, $\Lambda(n) \geq K n^{-1/d}$, one says the graph has $d$-dimensional isoperimetry.
The growth will be measured by $b_n = \displaystyle \sup_{x} |B_n(x)|$ where $B_n(x)$ is the ball of radius $n$ centred at $x$.

\begin{thm}\label{teo-intro-pol}~\\
\textbullet \; Let $G$ be a graph such that $\frac{K_1}{n^{1/d}} \leq  \Lambda(n) \leq \frac{K_2}{n^{1/d}}$ for some constants $K_1$ and $K_2$, then, $\exists K_3>0$ such that for all $n$,  $\dfrac{\Sep(n)}{n} \geq \dfrac{K_3}{n^{1/d}}$. \\
\textbullet \; Let $G$ be a graph such that $b_n \leq  K_1 n^d$ for some constant $K_1$, then, $\exists K_2>0$ such that for all $n$,  $\dfrac{\Sep(n)}{n} \leq \dfrac{K_2 \log n}{n^{1/d}}$. \\
\textbullet \; Let $G$ be a graph with $(1+\epsilon)$-dimensional isoperimetry and $b_n \leq  K_1 n^d$ then, for any $\delta >d$ there are constants $K_1$ and $K_2$ such that, for all $n$, $\dfrac{K_1 n^{\epsilon (1+\epsilon)/\delta^2}}{\log n} \leq \Sep(n) \leq K_2 n^{(d-1)/d} \log n$\\
\end{thm}
See Proposition \ref{tseppol-p}, Proposition \ref{tseppolgen-p} and Corollary \ref{tbornsuppol-c} for details.

This theorem can be used on Cayley graphs of nilpotent groups (for which a sharper upper bound was already given by Hume, Mackay \& Tessera \cite{HMT}), but our method applies also to other type of graphs, such as pre-fractal Sierpinski carpets (on this subject, see Gibson \& Pivarski \cite{GP} for isoperimetry and  Gladkova \& Shum \cite{GS} for a study of the relationship between conformal dimension and separation profile in graphs of fractals).

Our methods also yield results on the infinite percolation components of $\Z^d$, and more generally on a large class of graphs of polynomial growth, called polynomial graphs. Roughly speaking, a \emph{$ \left(d_1,d_2\right) $-polynomial graph }is a graph of growth bounded by $ n^{d_2} $ and of isoperimetric dimension at least $ d_1 $, see Definition \ref{polygraph} for details.
Since the percolation component always includes arbitrary large balls, it is more interesting to introduce a local variant of the separation profile in this context, namely the \emph{local separation at $v$}, where $v$ is a vertex of the graph (see Appendix \ref{appendix_local_sep}):
\[
\Sep_{G}^{v}(n) \vcentcolon = \sup_{F < B_G(v,r) ; \left| B_G(v,r) \right| \leq n } |F| \cdot h(F)
\]

In that case, we show that $\frac{\Sep^{v}(n)}{n}$ is bounded below by a function of the type $\tfrac{1}{n^\alpha}$, for every vertices in the polynomial case:

	\begin{thm}
	Let $ G $ be a $ \left(d_1,d_2\right) $-polynomial graph. Then for any $ \eta \in \left(0,1\right) $ there exists $ c>0 $ such that for any vertex $ v $ and any integer $ n $:
	
	\[ \Sep^v(n) \geq c n^{( 1 - \eta){\frac{d_1^2(d_1 - 1)}{d_2^3}}} \]
	
\end{thm}

 (see Theorems \ref{thm_poly1} and \ref{thm_poly2} for details), and for vertices that stay exponentially close to the origin in the $\mathbb{Z}^d$ percolation case:

 \begin{thm}
 	Let $\mathcal{C}_\infty$ be a supercritical phase percolation cluster of $ \mathbb{Z}^{d} $. Then for any $ \varepsilon\in \left(0,1\right)$ There exists almost surely $c>0$ such that for $ n $ large enough, if $\|x\|_\infty\leq \exp\left(n^{(1 - \varepsilon)\frac{d}{d - 1}}\right)$, then we have:
 	\[\Sep_{\mathcal{C}_\infty}^{x}(n) \geq c n^{\frac{d-1}{d}}\]
 \end{thm}

  (see Theorem \ref{thm_perc}). In this case, the inclusion in $\Z^d$ shows that this lower bound is optimal.
\\
\\
The same methods as those of Theorem \ref{teo-intro-pol} can also be applied to groups of intermediate growth:
\begin{thm}
 Let $G$ be a Cayley graph of a groups of intermediate growth with $K_1 e^{n^a} \leq b_n \leq K_2 e^{n^b}$ where $K_1,K_2 >0$ and $a,b \in ]0,1[$. Then\\
 \textbullet \; $\exists K_3 >0$ such that there are infinitely many $n$ with $\dfrac{\Sep(n)}{n} \geq 
  \dfrac{K_3}{(\log n)^{1+\frac{1}{a}}}$.\\
 \textbullet \; $\exists K_4>0$ such that for any $n$, $\dfrac{\Sep(n)}{n} \leq \dfrac{K_4}{(\log n)^{\frac{1}{b}-1}}$. 
\end{thm}

See Proposition \ref{tborninflog-p}, Remark \ref{remint} and Corollary \ref{tbornsupint-c} for details. The upper bound is obtained here using the growth assumption. Relations between growth and separation is studied in subsection \ref{ssgrow}.
\\
\\
Inside the realms of Cayley graphs, the lower bounds obtained are listed in the upcoming theorem.
(The list is not exhaustive, one could also get a lower bound for any group whose isoperimetric profile is known, \eg $\Z \wr \Z$.) Note that it is reasonable to compare $\frac{\Sep(n)}{n}$ with $\Lambda(n)$ for two reasons. First, for nilpotent groups those two functions coincide up to some multiplicative constants (see Hume, Mackay \& Tessera \cite{HMT}). Second, the underlying mechanism which allow us to provide such bounds relies on the known estimates of $\Lambda(n)$.

\begin{thm}\label{thm_log_decay} ~
Let $ G $ be a Cayley graph.

\centerline{
\begin{tabular}{|c|c|}
\hline
If, for some $a>0$, $\Lambda(N)$ is ... & \parbox[c][1.1cm][c]{7.5cm}{then, for infinitely many $N's$,  $\displaystyle \frac{\Sep(N)}{N}$ is ...}\\
\hline
$\simeq \dfrac{1}{\log(N)^{a}}$
& \parbox[c][1.2cm][c]{4.5cm}{$\succcurlyeq \dfrac{\Lambda(N)}{\log(N)} \simeq \dfrac{1}{\log^{1+a}(N)}$} \\
\hline
$\simeq \dfrac{1}{\log^{a}\big(\log(N)\big)}$
& \parbox[c][1.75cm][c]{7cm}{$\succcurlyeq \dfrac{\Lambda(N)}{\log(N)^C} \simeq \dfrac{1}{\log^{a}(\log(N)) \cdot \log(N)^C}$ (for some $C$)}\\
\hline
$\preceq \dfrac{1}{(\log N )^{a}}$
& \parbox[c][1.75cm][c]{7cm}{$\succcurlyeq \dfrac{\Lambda(N)}{\log(N)} $ }\\
\hline
$\preceq \dfrac{1}{\left(\log \dots \log \log N\right)^{a}}$
& \parbox[c][1.75cm][c]{7cm}{$\succcurlyeq \dfrac{\Lambda(N)}{N^{\eps}} $ , where $ \eps $ can be arbitrarily small}\\
\hline
\end{tabular}
}
\end{thm}
These estimates on the isoperimetric profile are known for polycyclic groups which are not nilpotent (first row of the table with $a=1$), wreath products $F \wr N$ where $F$ is finite and $N$ is a nilpotent group whose growth is polynomial of degree $d$ (first row with $a= 1/d$), iterated wreath products $F \wr (F \wr N)$ where $F$ is finite and $N$ is a nilpotent group whose growth is polynomial of degree $d$ (second row with $a = 1/d$), solvable groups in general (third row, see \cite{SF_Z} and \cite{SF_Z_isop}). See Proposition \ref{tborninflog-p}, Proposition \ref{tborninfloglog-p} and \S \ref{single_bound} for details.

Using the last row of the table, we prove the following theorem:
\begin{thm}\label{tgrores}
	Let $ G $ be a finitely generated solvable group. If there exists $ \epsilon \in \left(0,1\right) $ and $ c > 0 $ such that for any large enough integer $ n $ we have:
	\[ \Sep_{G}(n) \leq c n^{1 - \epsilon} \]
	Then $ G $ is virtually nilpotent.	
\end{thm}
See Theorem \ref{thm_gen_res} for details. Here is a nice application of Theorem  \ref{thm_gen_res}: (see Definition \ref{dregular} for the definition of a regular map)
\begin{cor}
	Let $ G $ be a finitely generated solvable group. If there exists a regular map from $ G $ to $ \mathbb{H}^{d} $ (the $ d-$dimensional hyperbolic space), then $ G $ is virtually nilpotent.
\end{cor}
\begin{remk}
	In the context the current paper, limited to graphs, it might not be clear what a regular map to $ \mathbb{H}^{d} $ is. Either replace $\mathbb{H}^d$ by any uniform lattice or see Hume, Mackay \& Tessera \cite{HMT} for the generalisation of the separation profile to Riemannian manifolds with bounded geometry.
\end{remk}
It might be disappointing to see that our lower bound is sometimes much smaller that the original isoperimetric profile (\eg a power of $\log N$ is much larger than $\log \log N$). However, our upper bounds show that such a dramatic loss cannot be avoided (see Theorem \ref{tcomp-c}).

In fact, for polycyclic groups, free metabelian groups, lamplighters on $\Z$ with finite lamps, lamplighters on $\Z^2$ with finite lamps and some iterated wreath products such as $F \wr (F \wr \Z)$ our methods show that $\frac{\Sep(n)}{n}$ is infinitely often $\geq K_1 (\log n)^{-c_1}$ while it is (for all $n$) $\leq K_2 (\log n)^{-c_2}$. 

More precisely,
$c_2 <1$ can be arbitrarily close to 1 (but there is no control on the constant $K_2$ as far as we know)
Furthermore, for polycyclic groups and lamplighters on $\Z$ with finite lamps $c_1 = 2$. 
For lamplighters on $\Z^2$ with finite lamps $c_1 = \frac{3}{2}$.
For free metabelian groups $c_1 > 1 + \tfrac{1}{r}$ is arbitrarily close to $1 + \tfrac{1}{r}$, where $r$ is the rank of the group.

The case of $F \wr (F \wr \Z)$ is of particular interest, since it shows that the appearance of the logarithmic factor in the lower bound is necessary (see Remark \ref{remlog}). This also shows that there are amenable groups for which $\frac{\Sep(n)}{n}$ decays much faster than $\Lambda(n)$.

Though our lower bounds apply to a vast array of groups, we show, using constructions of Erschler \cite{Ers2} and Brieussel \& Zheng \cite{BZ}, that there are groups for which those methods fail to give a significant bound.

Our upper bounds on the separation profile come either from the growth of balls (for graphs which do not have exponential growth, see Theorems \ref{tbornsupint-c} ans \ref{tbornsuppol-c}) or from the compression exponents:

\begin{thm} Let $G$ be a bounded degree graph which has compression exponent (in some $L^p$-space) equal to $\alpha >0$. 
Then, for any $c < \frac{\alpha}{2-\alpha}$, there is a constant $K$ so that, for any $N$, $\displaystyle \frac{\Sep(N) }{N} \leq \frac{K}{(\log N)^c}$.
\end{thm}

See Corollary \ref{tcomp-c} for details.

In addition to the afore-mentioned examples, the previous theorem applies to hyperbolic groups (these have $\alpha =1$). Because $\alpha(G \times H) = \min(\alpha(G), \alpha(H))$, this shows that the separation profile of the product of two hyperbolic groups satisfies $\frac{\Sep(n)}{n} \leq \frac{K_1}{(\log n)^{1 - \eps}}$ (for any $\eps\in\left(0,1\right)$). This is quite sharp since the product of two trees has $\frac{\Sep(n)}{n} \geq \frac{K}{(\log n)}$ (see Benjamini, Schramm \& Timár \cite[Theorem 3.5]{BST}).

{\it Organisation of the paper:} \S{}\ref{sdefi} contains the basic definitions. In \S{}\ref{sopti}, we make the first step towards a lower bound by looking at sets which have a small boundary to content ratio. These optimal set turn out to have a high separation. This estimate is then used in \S{}\ref{slow} to get lower bounds on the separation profile from the isoperimetric profile. \S{}\ref{sappli} is devoted to concrete estimates in various Cayley graphs and self-similar graphs (\S{}\ref{sssfrac}). 
The proof of Theorem \ref{tgrores} (as well as some further lower bounds) is done in \S{}\ref{single_bound}.
Groups where the methods do not yield a lower bound are constructed in \S{}\ref{sslimit}. 
Upper bounds on the separation profile are done in \S{}\ref{suppe}: via growth in \S{}\ref{ssgrow} and via compression in \S{}\ref{sscomp}. 
Questions are presented in \S{}\ref{sques}. 
Lastly, an appendix is devoted to the study of local separation profiles, with applications to infinite percolation component in $\Z^d$ (\S{}\ref{sspoly1}) and to polynomial graphs (\S{}\ref{sspoly1} and \ref{sspoly2} ).

{\it Acknowledgments:} The authors would like to thank 
Jérémie Brieussel for having initiated and followed this project all along; Romain Tessera for various interesting questions, remarks and corrections; Itai Benjamini who proposed us to study local separation of polynomial graphs; and Todor Tsankov for noticing a mistake in a previous version of Theorem \ref{thm_perc}.

%____________________
%____________________
\section{Definitions}\label{sdefi}
%____________________

The set of vertices of a graph $G$ will be denoted $VG$, while the set edges will be written $EG$. 
The set of edges consists in ordered pairs of vertices: $EG \subset VG \times VG$. 
Nevertheless, one may notice that the separation profile is monotone under quasi-isometric embeddings only for graphs on bounded degree.

\begin{defn}[Boundary]
Let $G$ be a graph. For any subset $A\subset VG$, its {\bf boundary} is the set $\partial A= \{ (a,b) \in EG, a\in A \Leftrightarrow b\in A^c \}$.
\end{defn}

\begin{defn}[Isoperimetric profile]
The {\bf isoperimetric profile} of a graph $G$ is the function $\Lambda_G: \N \to \R_{\geq 0}$ defined by
$$\Lambda_G(n) = \inf_{A\subset VG, |A|\leq n} \dfrac{|\partial A|}{|A|} $$
\end{defn}

Note that, for an infinite connected graph, the isoperimetric profile never takes the value 0.

\begin{defn}[Cheeger constant]
For any finite graph $G$, the {\bf Cheeger constant} $h(G)$ of $G$ is 
$$
h(G) = \min_{A\subset G, |A|\leq \frac{|G|}{2}} \dfrac{|\partial A|}{|A|}
$$
\noindent
Let $\Gamma$ be an infinite graph. 
For any finite subset $F\subset V\Gamma$, let $\widetilde{F}$ be the subgraph of $\Gamma$ induced on $F$:
\begin{itemize}
\item $V \widetilde{F} = F $
\item $E \widetilde{F} = \{ (a,b) \in V\Gamma \mid a,b\in F \} $
\end{itemize}
By a small abuse of notation, the Cheeger constant of $F\subset V\Gamma$ is $h(F) = h(\widetilde{F})$.
\end{defn}

\begin{defn}[Separation profile]
Let $\Gamma$ be an infinite graph of bounded degree.
The {\bf separation profile} of $\Gamma$ is the function $\Sep: \N \to \R_{\geq 0}$ defined by 
$$ \Sep (n) = \sup_{F\subset V\Gamma, |F|\leq n} |F| \cdot h(F)
$$
\end{defn}

As remarked by David Hume in \cite{HMT}, it comes naturally from \cite[Proposition 2.2. and Proposition 2.4.]{H} that this definition is equivalent to the original one from Benjamini, Schramm \& Timár \cite{BST}.

One may notice that we use here the edge-boundaries, unlike Hume who uses vertex-boundaries. However, under the assumption that the graph has a bounded degree those two differ only by a constant factor.

\begin{remk}
Since any graph with an infinite connected component contains a half-line, there is a trivial lower bound on $\Sep(n)$: $\Sep(n) \geq 1$ or $\frac{\Sep(n)}{n} \geq \frac{1}{n}$. Recall that Benjamini, Schramm \& Timàr \cite[Theorem 2.1]{BST} showed that a graph with bounded separation admits a regular map into a tree.
\end{remk}

% ____________________________________
\section{Lower bound on the separation from isoperimetry}

% ____________________________________
\subsection{Optimal sets and their Cheeger constant}\label{sopti}
% ____________________________________

This section is devoted to the following question: given an infinite graph that has a rather large isoperimetric profile, does the same holds for the Cheeger constants of its finite subgraphs ?

There are two simple negative answers. 
First, in any infinite connected graph, there is an infinite subset $L$ so that the graph induced to $L$  is an infinite half-line.
However, a finite subgraph of the half-line has the weakest Cheeger constant.
Clearly, one needs to restrict a bit more the sets considered.
It turns out the right thing to do is to restrict only to sets $F$ which are ``optimal'' for the isoperimetric problem.

Second, consider the infinite regular tree; a graph with strong isoperimetric profile (\ie the isoperimetric profile is bounded from below by a constant).
Any graph induced by a finite subset has (again) the weakest possible Cheeger constant.

The aim of this section is to show that in graphs without strong isoperimetric profiles, the Cheeger constant on the optimal induced subsets is still fairly strong. 

\begin{defn}[Optimal sets and integers]
Let $\Gamma$ be an infinite graph
\begin{itemize}
\item A subset $F$ of $V\Gamma$ is called {\bf optimal} if $\dfrac{|\partial F|}{|F|} = \Lambda(|F|)$,
\emph{i.e.}: 
$$\forall G\subset V\Gamma, \qquad |G| \leq |F| \Rightarrow
\dfrac{|\partial G|}{|G|} 
\geq 
\dfrac{|\partial F|}{|F|}.$$

\item An integer $n$ will be called {\bf optimal} if there exists an optimal subset of cardinality $n$.

\end{itemize}
\end{defn}

\begin{lem}\label{tbase-l}
Assume $F$ is optimal. Then:
$$
2 h(F) \geq \Lambda\left(\frac{|F|}{2} \right) - \Lambda\left(|F|\right)
$$
\end{lem}
\begin{proof}
The Cheeger constant for the [finite] graph induced on $F$ is given by looking at subsets $F_1$ of $VF$ of size at most equal to $\frac{|F|}{2}$, and trying to minimise $\frac{|\del_F F_1|}{|F_1|}$, where $\del_F$ denotes the boundary in $F$.

Let $F_1$ be a subset of $VF$ of size at most equal to $\frac{|F|}{2}$, let $F_2 = F\setminus F_1$. For any (disjoint) subsets $ A $ and $ B $ of $ V\Gamma $, we denote by $ E(A,B) $ the set of edges of $ \Gamma $ that have one endpoint in $ A $, and the other in $ B $. We have:
\begin{align*}
	\partial F_1 &= E\left(F_1 , V \Gamma \setminus F_1 \right)\\
	&= E\left(F_1 , V \Gamma \setminus F \right) \sqcup E\left(F_1,F_2 \right)\\
	&= E\left(F_1 , V \Gamma \setminus F \right) \sqcup \partial_F F_1
\end{align*}
Similarly,
\begin{align*}
	\partial F_2 &= E\left(F_2 , V \Gamma \setminus F \right) \sqcup \partial_F F_2\\
	&= E\left(F_2 , V \Gamma \setminus F \right) \sqcup \partial_F F_1\\
\end{align*}
Then we have
\[ 2|\del_F F_1| = |\del F_1| + |\del F_2| - |\del F| \]
Moreover, we have $
\Lambda(|F_1|)\geq \Lambda(|F|/2)
$, and as $F$ is optimal,
$
\frac{|\del F_2|}{|F_2|} \geq \frac{|\del F|}{|F|}
$
and
$
\Lambda(|F|) = \frac{|\del F|}{|F|}
$

Using these facts, we can deduce that
\begin{align*}
2|\del_F F_1| 
&\geq |F_1|\Lambda(|F_1|) + |F_2|.\frac{|\del F|}{|F|} - |F|.\frac{|\del F|}{|F|} \\
&= |F_1|\Lambda(|F_1|) - |F_1|.\frac{|\del F|}{|F|}\\
&\geq |F_1|\left(\Lambda(|F|/2) - \Lambda(|F|)\right)
\end{align*}

Since this is true for any $F_1\subset VF$ of size at most $|F|/2$, this concludes the proof.
\end{proof}

One already sees that these methods give basically nothing in graphs with a strong isoperimetric profile, since $\Lambda$ is nearly constant. 
This is however to be expected since there can be no general reasonable bound in this family of graphs (the typical example would be infinite $k$-regular trees).

\begin{cor}
If $n>0$ is optimal, then:
$$
2\dfrac{\Sep(n)}{n} \geq \Lambda(n/2) - \Lambda(n)
$$
\end{cor}

\begin{proof}
Assume $F$ is optimal of cardinality $n$. 
Then $2 \frac{\Sep(n)}{n} \geq 2 h(F) \geq \Lambda(n/2) - \Lambda(n)$
\end{proof}

% _________________________________________________________________
% _________________________________________________________________
\subsection{A lower bound on the separation profile from isoperimetry}\label{slow}
% _________________________________________________________________

The following theorem is a consequence of the previous lemma and it applies to a large class of graphs, providing that the isoperimetric profile tends to 0 without reaching it (equivalently, the graph is amenable and has no finite connected components). It can be simplified in the case of graphs with ``many symmetries'', see Theorem \ref{tpropsep-t}. 

\begin{thm}\label{tpropnonsym-t}
Let $G$ be an infinite connected amenable graph of bounded degree.
Assume there is an increasing function $p:\N \to \N$ so that for any $n$ there is a $k \in (n, p(n)]$ such that $k$ is optimal. 
Choose $\eps \in (0,1) $.
Let $n\geq 1$ be an integer.
Let $m\geq 1$ be such that $\Lambda(m) \leq (1-\eps)\Lambda(n) $.

Then there exists an $N\in [n,p(m)]$ such that 
$$
\boxed{
\dfrac{\Sep(N)}{N} \geq \eps\dfrac{\Lambda(n)}{4 \log(\frac{p(m)}{n}) + 4}.
}
$$
\end{thm}

For example, it applies to graphs where the isoperimetric profile is bounded above and below:
$\frac{C_1}{n^a} \leq \Lambda(n) \leq \frac{C_2}{n^a}$.
Given an optimal integer $n_o$ the next optimal integer $n_p$ happens at the latest when 
$\frac{C_1}{n_o^a} \geq \frac{C_2}{n_p^a}$, 
or in other words 
$n_p \leq (\frac{C_2}{C_1})^{1/a} n_o$
(and the function $p$ is linear). We study this specific case in Theorem \ref{tseppol-p}

\begin{proof}
Let $n_0 = \min\{ k\geq n \mid k \text{ is optimal} \}$.
We define recursively $i\in\N$, $n_{i+1} = \max \{k\in (n_i,2n_i] \mid k $ is optimal $\}$, if this set is non-empty, and, otherwise, $n_{i+1} = \min \{ k \mid k \in [2n_i, p(n)]$ and $k$ optimal$ \}$.

Remark that $\forall i, n_{i+2} \geq 2 n_i$.
Let $i_{\max}$ be the first index $i$ for which $n_i \geq m$. Then
$$n_{i_{\max} - 1} \leq m \leq n_{i_{\max}} \leq  p( n_{i_{\max} - 1}) \leq p(m).$$

Using Lemma \ref{tbase-l},
$$
\forall i\in [0,i_{\max}], \qquad 2\dfrac{\Sep(n_i)}{n_i} \geq \Lambda(n_i/2) - \Lambda(n_i)
$$
Summing up all these inequalities, one gets
\[
2\sum_{i = 0}^{i_{\max}} \frac{\Sep(n_i)}{n_i} 
\geq \Lambda(n_0/2) - \Lambda(n_{i_{\max}}) 
+ \sum_{i = 0}^{i_{\max} - 1} \Lambda(n_{i+1}/2) - \Lambda(n_{i}).
\]
Either $\tfrac{1}{2} n_{i+1} \leq n_i$ (so that $\Lambda(n_{i+1}/2) \geq \Lambda(n_{i})$ since $\Lambda$ is decreasing) or $\Lambda(n_{i+1}/2) \geq \Lambda(n_{i})$ because $n_{i+1}$ is the next optimal integer after $n_i$. 
Either way, the sum on the right-hand side is positive, and consequently
\[
\begin{array}{r@{\;\geq\;}l}
\displaystyle 2 \sum_{i = 0}^{i_{\max}} \frac{\Sep(n_i)}{n_i} 
& \Lambda(n_0/2) - \Lambda(n_{i_{\max}})  \\
& \Lambda(n) - \Lambda(m)  
\\
& \eps \Lambda (n)
\end{array}
\]

Since $m \leq n_{i_{\max}}$ (hence $\Lambda(m) \geq \Lambda(n_{i_{\max}})$)  and $\Lambda(n_0/2)\geq \Lambda(n)$ because:
\begin{itemize}
\item either $n_0 \leq 2n$, and since $\Lambda$ is non-increasing $\Lambda(n_0/2)\geq \Lambda(n)$
\item either $n_0 \geq 2n$, therefore $\lfloor n_0/2 \rfloor$ is not optimal, so $\Lambda(n_0/2)=\Lambda(n)$.
\end{itemize}

From there, we can deduce that
$$
\exists j\in [0,i_{\max}], \qquad \frac{\Sep(n_j)}{n_j} \geq \dfrac{\eps }{2}\dfrac{\Lambda (n)}{i_{\max}+1}
$$
But recall that $\forall i, n_{i+2}\geq 2 n_i$. 
Consequently, $n_{i_{\max}} \geq 2^{\lfloor \frac{i_{\max}}{2} \rfloor } n_0 $. 
Furthermore, $n_{i_{\max}} \leq p(m) $. 
This implies: 
$2^{\lfloor \frac{i_{\max}}{2} \rfloor } n_0 
\leq 
p(m)$. 
Since $n_0 =n$,
$2^{\lfloor \frac{i_{\max}}{2} \rfloor } \leq 
\frac{ p(m) }{n}$
Thus, 
\[
i_{\max} + 1 
\leq
2 \lfloor \frac{i_{\max}}{2} \rfloor + 2 \leq 2\log_2 \Big(\frac{ p(m) }{n} \Big) + 2
\]
The claim follows if we choose $N \vcentcolon = n_j$.
\end{proof}

\begin{defn}
Let us say a graph $G$ has {\bf partial self-isomorphisms}, if, for every finite set $F \subset VG$, there exists another finite $F'$ such that $F \cap F' = \varnothing$ and the graph induced on $F \cup \del F$ is isomorphic (as a finite graph) to $F' \cup \del F'$ 
\end{defn}

Note that having partial self-isomorphisms implies the graph is infinite.
This property is satisfied by fairly natural classes of graphs such as Cayley graphs, graphs with vertex-transitive (or edge-transitive) automorphisms and self-similar graphs.

\begin{lem}\label{tnonvid-l}
Let $G$ be a graph which has partial self-isomorphisms. Assume $n$ is an optimal integer. The set $\{ k \in (n,2n] \mid k $ is optimal$ \}$ is not empty.
\end{lem}
\begin{proof}
If there is an optimal subset $F$ whose size is $(n,2n]$, there is nothing to prove. 
Otherwise, let us construct such an optimal set of size $2n$.

Let $F$ be an optimal set of size $n$.
Using the property of partial self-isomorphisms, there is a set $F'$ with $|F|=|F'|$, $|\partial F|=|\partial F'|$, and $F \cap F' = \varnothing$.
Hence $|F\cup F'| = 2n$ and $\dfrac{|\partial (F\cup F')|}{|F\cup F'|} \leq \dfrac{|\partial F|}{|F|}$. 
Since we assumed there are no optimal sets whose size is in $(n,2n]$, then for any $G \subset \Gamma$ such that $|G|\leq 2n$, we have
$
\dfrac{|\partial G|}{|G|} 
\geq 
\dfrac{|\partial F|}{|F|} 
\geq 
\dfrac{|\partial (F\cup F')|}{|F\cup F'|}
$. 
Therefore $F\cup F'$ is optimal.
\end{proof}

Since there is always an optimal integer ($n=1$!), any graph with partial self-isomorphisms always has infinitely many optimal $n$.

\begin{remk}
	Even without the assumption that the graph has partial self-isomorphisms, it is still possible to get some information on optimal integers, using the bounds on the isoperimetric profile. We still then obtain bounds on the separation profile, a priori worse and possibly trivial. However, in Theorems \ref{tseppol-p} and \ref{tseppolgen-p} where we study the case where the isoperimetric profile is bounded above and below by powers of $ n $, it turns out that we get interesting bounds using this fact.
	
	This strategy is very adapted to polynomial graphs, so we applied it also	in the proof of Theorem \ref{thm-gen-loc}, taking a function $ p(n) $ satisfying the more restrictive condition $ \Lambda(p(n)) \leq \Lambda(n)/2 $. It simplifies the proof without any loss.
	
\end{remk}

\begin{thm}\label{tpropsep-t}
Assume $\Gamma$ is a connected amenable graph of bounded degree with partial self-isomorphisms.
Let $n\geq 1$ and $\eps \in (0,1) $. 
Let $m\geq 1$ be such that $\Lambda(m) \leq (1-\eps)\Lambda(n) $.

Then there exists an $N\in [n,2m]$ such that 
$$
\boxed{
\dfrac{\Sep(N)}{N} \geq \eps\dfrac{\Lambda(n)}{4 \log(m/n) + 8}.
}
$$
\end{thm}
\begin{proof}
This result comes naturally from Theorem \ref{tpropnonsym-t} and Lemma \ref{tnonvid-l}.
\end{proof}

%______________________
%______________________
\section{Applications}\label{sappli}
%______________________

In this section, we give applications of Theorems \ref{tpropnonsym-t} and \ref{tpropsep-t}. We use isoperimetric profiles that have already been computed in the literature. In the course of this investigation, there are three factors that come into play:

\begin{itemize}
\item The geometry / the symmetries of the graph: the function $p(n)$ of Theorem \ref{tpropnonsym-t}.
\item The decay of the isoperimetric profile.
\item Inaccurate knowledge of the isoperimetric profile: when we only have loose bounds on the isoperimetric profile.
\end{itemize}

Our goal on this section is not to give an exhaustive overview of possible applications, but only to apply Theorem \ref{tpropnonsym-t} in situations that seemed interesting to us. In \S{}\ref{sslimit} we look at limit cases, where this theorem gives no information.

Bendikov, Pittet \& Sauer \cite[Table 1 on p.52]{BPS} contains many reference for the isoperimetric profile of groups. As noted in Erschler \cite[\S{}1]{Ers} the isoperimetric profile $\Lambda$ is connected to the F{\o}lner function $F$ by the relation: 
\[
 \Lambda(N) \simeq \frac{1}{F^{-1}(N)}
\]

\subsection{Isoperimetric profile decaying as a power of $N$}\label{sspown}

Recall that virtually nilpotent groups (equivalently groups for which the cardinality of a ball of radius $r$ is bounded by polynomials in $r$) are the only groups where $\Lambda(n)$ is of the form $\tfrac{1}{n^{1/d}}$ (where $d$ is the degree of the polynomial); see Pittet \& Saloff-Coste \cite[Theorem 7.1.5]{PSC-surv} or \cite[Theorem 3.4]{PSC-conf}.

A polynomial upper bound on the isoperimetric profile is given by Benjamini \& Papasoglu \cite[Theorem 2.1]{BP} for doubling planar graphs.

\begin{prop}\label{tseppol-p}
Let $G$ be a graph of bounded degree such that if $n$ is large enough, the following inequality holds:
\[
\tag{\S{}}\label{isop_polyn}\dfrac{C_1}{n^\beta} \leq \Lambda_G(n) \leq \dfrac{C_2}{n^\beta}
\] for some constants $C_1, C_2, \beta >0$. 

Then if $n$ is large enough: 
$$\dfrac{\Sep(n)}{n} 
\geq 
A \cdot \Lambda(n)
$$ for some constant $A$.
\end{prop}

One may notice that this proposition applies for very general graphs (although the hypothesis on $G$ imply that it is amenable and has no finite connected component).

\begin{proof}
Let $n_0$ be an integer such that \eqref{isop_polyn} holds for any $n\geq n_0$.
For any such $n$,
$$\Lambda\left(\left(\frac{C_2}{C_1}\right)^{1/\beta} \cdot n\right)\leq \frac{C_1}{n^{\beta}}\leq\Lambda(n) $$

Therefore with the notations of Theorem \ref{tpropnonsym-t}, we can take $p(n) = \left(\frac{C_2}{C_1}\right)^{1/\beta} \cdot n$.

Let $C$ be the smallest integer larger than $\left( \dfrac{C_2}{\frac{1}{2}\cdot C_1}\right)^{1/\beta}$.
Let $n$ be an integer such that $n \geq n_0$ and let $m$ be the smallest integer such that $m \geq Cn$.
Since $m^\beta \geq 2 \dfrac{C_2}{C_1} n^\beta$, it follows that
$\Lambda(m)
\leq
\dfrac{C_2}{m^\beta}
\leq
\dfrac{\frac{1}{2} \cdot C_1}{n^\alpha}
\leq
\frac{1}{2}\Lambda(n)
$.

Then we can apply Theorem \ref{tpropnonsym-t}: there is a $N\in [n,2m]$ such that 
\[
 \dfrac{\Sep(N)}{N} 
\geq 
\frac{\eps}{4} \dfrac{\Lambda (n)}{\log\left(\frac{p(m)}{n}\right) + 1}
=
\frac{\eps}{4} \frac{\Lambda (n)}{
\log\left(\left(\frac{C_2}{C_1}\right)^{1/\beta} \cdot \frac{m}{n} \right) + 1}
\]

Since $m/n - 1 \leq C$, and $\log(m/n)\leq \log(m/n-1) +1 \leq \log(C) + 1$, finally we get: 
\[
 \dfrac{\Sep(N)}{N} 
\geq
K \Lambda (n)
\]
with $K=\frac{\eps}{4 \log\left(\left(\frac{C_2}{C_1}\right)^{1/\beta}\right) + \log(C) + 1}$. 

If we additionally suppose that $n\geq  4C^2$, then we have:
\[\frac{\Sep(n)}{n} \geq \frac{K}{2C+1}\frac{1}{n^\beta} \]

Indeed, assume $n  \geq n_0$ and $n\geq 4C^2$. 
We know that there exists an integer $N\in\left[\lfloor \frac{n}{2C} \rfloor , n \right]$ such that $\frac{\Sep(N)}{N}\geq K\Lambda(n)$.
\\
Then we have:
\[\frac{\Sep(n)}{n}
\geq
\frac{\Sep(N)}{(2C+1)N}
\geq
\frac{K}{2C+1} \cdot \frac{1}{N^\beta}
\geq
\frac{K}{2C+1} \cdot \frac{1}{n^\beta}
\geq
\frac{K}{(2C+1) \cdot C_2}  \cdot \Lambda_G(n)
\] 
This concludes the proof.
\end{proof}

Note that in the case of virtually nilpotent groups, and more generally for vertex-transitive graphs with polynomial growth, Hume, Mackay \& Tessera \cite[Theorem 7]{HMT} show the inequality of Proposition \ref{tseppol-p} is sharp.

\begin{remk}
If we assume that the graph has partial self-isomorphisms then we can take $p(n)=2n$ according to Lemma \ref{tnonvid-l}. Therefore we may improve the constant $K_1$ of Proposition \ref{tseppol-p}.
\end{remk}

In the spirit of Benjamini \& Schramm \cite{BS}, we can deduce the following corollary under very minimal assumptions on the graph:

\begin{prop}\label{tseppolgen-p} 
If one assumes that for a graph of bounded degree there exists $C_1, C_2, \alpha, \beta>0$ such that for any positive integer $n$, we have
$\dfrac{C_1}{n^\alpha} 
\leq 
\Lambda_G(n) 
\leq 
\dfrac{C_2}{n^\beta}$ and $1 > \alpha > \beta$,
then there exists  $A>0$ such that for any $n>0$ we have:
\[\Sep(n) \geq A \cdot \frac{n^\gamma}{\log(n)}\]
with:
\begin{itemize}
\item $\gamma = \frac{\beta \left(1 - \alpha \right)}{\alpha} $ if $G$ has partial self-isomorphisms
\item $\gamma = {\frac{\beta^2 \left(1 - \alpha \right)}{\alpha^2}}$ otherwise.
\end{itemize}
\end{prop}

\begin{proof}
Without assuming partial self-isomorphisms, we can apply Theorem \ref{tpropnonsym-t}, with $p(n)\simeq n^{\frac{\alpha}{\beta}}$ and $m \simeq n^{\frac{\alpha}{\beta}}$. Then for any integer $n$ we have an integer $N\in\left[n,Cn^{\left(\frac{\alpha^2}{\beta^2}\right)}\right]$ such that $\Sep(N)\succeq \displaystyle \frac{N \cdot \Lambda(n)}{\log(n)}$ Now, let $k$ be an positive integer. Let $n=\left(\frac{k}{C}\right)^{\frac{\beta^2}{\alpha^2}}$. Then there exists some $N\in\left[n,k\right]$ such that 
\[
\Sep(N)\succeq \frac{N \cdot \Lambda(n)}{\log(n)} \succeq \frac{n}{n^{\alpha}\log(n)} = \frac{n^{1-\alpha}}{\log(n)} \simeq \frac{k^{\gamma}}{\log(k)}   
\]
with $\gamma = \dfrac{\beta^2 \left(1-\alpha\right)}{\alpha^2}$. Since $\Sep(k)\geq \Sep(N)$, we get the announced lower bound.

If the graph has partial self-isomorphisms, then Lemma \ref{tnonvid-l} shows that $p(n)=2n$ is a valid choice and the rest of the proof is similar.
\end{proof}

\subsubsection{Application to pre-fractal Sierpinski carpets}\label{sssfrac}

Gibson \& Pivarski showed in \cite{GP} some results on isoperimetry in pre-fractal graphical Sierpinski carpets.

Pre-fractal Sierpinski carpets are built using an iterating process. We consider a squared fundamental domain $F_1$ which is a union of little squares, obtained by removing subsquares in an admissible way. We can consider $F_1$ as a pattern. We make copies of $F_1$ in such a way that we reproduce this pattern at a larger scale. We get a bigger square that we can call $F_2$. That is the first step of this process, and the pre-fractal Sierpinski carpet is the limit object that we get iterating the process indefinitely.

The carpet is then a subset of $\R^2$, which is a union of little squares.
The associated graph is obtained putting a vertex in the centre of each of these squares, and linking vertices with an edge if and only if their squares share a common face in the carpet.

We use the notations of \cite{GP}: $F_1$ is the fundamental domain of the pre-fractal, $m_F$ is the number of sub-squares in $F_1$ and $R$ is the number of columns of $F_1$ with one or more squares removed.

\begin{thm}[see Gibson \& Pivarski \cite{GP}] \label{fract1}
Let $X$ be a two-dimensional pre-fractal graphical Sierpinski carpet which satisfies the regularity assumptions and the sparse row property. Then
\[\Lambda_X(n) \asymp n^{\frac{\log(R)}{m_F}-1} \]

\end{thm}

\begin{proof}
The lower bound comes from \cite[Theorem 4.4]{GP}, together  with \cite[Corollary 3.2]{GP} and \cite[Corollary 3.8]{GP} to convert the result of the Theorem for the graphical isoperimetry.
\\
The upper bound comes from the construction of explicit subsets \cite[Lemma 4.1]{GP}, together with \cite[Lemma 3.3]{GP} (to get \emph{graphical} subsets) and \cite[Corollary 3.2]{GP}.\end{proof}

The construction of pre-fractal Sierpinski carpets can be generalised in higher dimensions. The following theorem holds for standard Sierpinski carpets of any dimension.

\begin{thm}[see Gibson \& Pivarski \cite{GP}] \label{fract2}
Let $X$ be the $n$-dimensional pre-fractal graphical standard Sierpinski carpet. Then
\[\Lambda_X(n) 
\asymp 
n^{-\frac
{\log\left(3^n-1\right)-\log\left(3^{n-1}-1\right)}
{\log\left(3^n-1\right)}
} \]

\end{thm}

\begin{proof}
The proof is very similar to the latest proof, using $n$-dimensional results: \cite[Corollary 4.6]{GP} and \cite[Corollary 4.2]{GP}.
\end{proof}

In this context, Proposition \ref{tseppol-p} applies, so we can deduce the following corollary:

\begin{cor}
Under the assumptions of Theorem \ref{fract1} or of Theorem \ref{fract2} , there exists  $n_0,K_1>0$ such that
$$
\forall n\geq n_0\ 
\dfrac{\Sep_X(n)}{n} 
\geq 
K_1 \cdot \Lambda(n)
$$
\end{cor}

\subsection{Isoperimetric profile with logarithmic decay} \label{sslog}
Before moving on to the next class of examples, let us recall that for polycyclic groups of exponential growth (as well as solvable groups with finite Prüfer rank and geometrically elementary solvable groups) the isoperimetric profile is known to be of the form $\dfrac{C_1}{\log (n)} \leq \Lambda_G(n) \leq \dfrac{C_2}{\log(n)}$; see Pittet \& Saloff-Coste \cite[Theorem 7.2.1]{PSC-surv}, \cite[Theorem 3.4]{PSC-conf} for polycyclic groups, and Bendikov, Pittet \& Sauer \cite[Table 1]{BPS}, Pittet \& Saloff-Coste \cite{PSC-solv} and Tessera \cite{Tes13} for more general statements.

Also a group of intermediate growth (\ie a group where the cardinality of balls are such that $e^{n^a} \preccurlyeq |B_n| \preccurlyeq e^{n^b}$) are known to have a bound  $\dfrac{C_1}{ (\log n)^{\tfrac{1}{a}}} \preccurlyeq \Lambda_G(n) \preccurlyeq \dfrac{C_2}{(\log n)^{\tfrac{1}{b} -1}}$; see Pittet \& Saloff-Coste \cite[Theorem 6.2.1]{PSC-surv}.

Lastly, wreath products $F \wr N$ where $F$ is finite and $N$ has polynomial growth of degree $d$ have an isoperimetric profile of the form 
$\dfrac{C_1}{ (\log n)^{\tfrac{1}{d}}} \preccurlyeq \Lambda_G(n) \preccurlyeq \dfrac{C_2}{(\log n)^{\tfrac{1}{d}}}$; see Pittet \& Saloff-Coste \cite[\S{}4 and \S{}7]{PSC-conf} or Erschler \cite[Theorem 1]{Ers}.

\begin{prop}\label{tborninflog-p}
Let $C_1, C_2, \alpha,\beta > 0$. 
Let $G$ be an infinite connected amenable graph of bounded degree with partial self-isomorphisms such that
$\frac{C_1}{\log^\alpha (n)} 
\leq 
\Lambda_G(n) 
\leq 
\frac{C_2}{\log^\beta (n)}$

Then there exists a constant $K_1$ such that for infinitely many $N$'s, the following inequality holds:
$$
\boxed{
\dfrac{\Sep(N)}{N} 
\geq 
K_1 \dfrac{\Lambda(N)}{\log^{\alpha/\beta}(N)}
}
$$
\end{prop}

\begin{proof}
Let $\eps \in (0,1)$,
$C = \left( \dfrac{C_2}{C_1 (1 - \eps)}\right)^{1/\beta}$, $n$ be a positive integer larger than $\exp\left({\left(\frac{3}{C}\right)^{\beta/\alpha}}\right)$ and $k=2m$, with $m$ being the smallest integer such that $m \geq \exp({C \log^{\alpha/\beta}(n)})$.

Under our hypothesis, 
$mn 
\geq 
m
\geq 
\exp\left({\left( \frac{C_2}{ C_1 (1 - \eps)} \right)^\frac{1}{\beta} \log^{\frac{\alpha}{\beta}} (n)}\right)$. Hence
$\log^\beta(mn) \geq \dfrac{C_2}{C_1 (1 - \eps)} \log^\alpha(n)$.

Consequently, $\Lambda(mn)
\leq
\dfrac{C_2}{\log^\beta (m n)}
\leq
\dfrac{C_1 (1 - \eps)}{\log^\alpha (n)}
\leq
(1 - \eps )\Lambda(n)
$.
Thanks to Theorem \ref{tpropsep-t}, there is a $N\in [n,kn]$ such that:
$$\dfrac{\Sep(N)}{N} 
\geq 
\eps \dfrac{\Lambda (n)}{4\log(m) + 8}
\geq
\eps \dfrac{\Lambda (n)}{8C\log^{\alpha/\beta}(n)}
$$

Putting $K_1 = \frac{\eps}{8 C}$ yields the first claim.
A simple calculus shows that this constant is optimal (\ie maximal) when $\eps = \frac{\beta}{\beta+1}$.
\end{proof}

\begin{cor}
Assume $G$ is a polycyclic group of exponential growth (or, more generally, a solvable group with finite Prüfer rank or a geometrically elementary solvable group), then there exists two constants $K_2$ and $k$ such that for any $n$ (if $n$ is large enough) there exists an integer $N\in\left[n,n^k\right]$ such that:
\[\dfrac{\Sep(N)}{N} 
\geq 
K_2 \dfrac{\Lambda(n)}{\log(n)}
\simeq 
\dfrac{1}{\log(n)^2}
\]
\end{cor}

\begin{cor}
Assume $G$ is a wreath product $F \wr N$ where $F$ is finite and $N$ has polynomial growth of degree $d$. 
Then there exists two constants $K_2$ and $k$ such that for any $n$ (if $n$ is large enough) there exists an integer $N\in\left[n,n^k\right]$ such that:
\[\dfrac{\Sep(N)}{N} 
\geq 
K_2 \dfrac{\Lambda(n)}{\log(n)}
\simeq 
\dfrac{1}{\log(n)^{\frac{d+1}{d}}}
\]
\end{cor}

Although it is unclear whether this lower bound (or that of Proposition \ref{tborninflog-p}) is sharp, let us note that if we assumed all $N$ were optimal and $\Lambda(N) = \frac{1}{\log(N)^d}$ then Lemma \ref{tbase-l} would only give $\frac{\Sep(N)}{N} \succcurlyeq \frac{1}{\log(N)^{d+1}}$ (for $N$ large enough). 
Hence, if the bound lacks sharpness in this case, then this loss already stems from Lemma \ref{tbase-l}.

Note that for the following corollary, better bounds could be obtained, if one had direct bounds on the isoperimetric profile $\Lambda$.
\begin{cor}\label{tborninfint-c}
Let $G$ be a group of intermediate growth with $e^{n^a} \preccurlyeq |B_n| \preccurlyeq e^{n^b}$. 
Then there are constants $K_0$, $K_1$, $K_2$ and $K_3 >0$ so that for any $n \geq K_0$, there is a $N \in [n,n^{K_1 \log^{K_2}(n)} ]$ with
\[
\dfrac{\Sep(N)}{N} 
\geq 
K_3 \dfrac{\Lambda(n)}{\log^{\frac{b}{a(1-b)}}(N)}  \succcurlyeq \frac{1}{\log^{\frac{1}{a(1-b)}}(N)}
\]

\end{cor}
See Remark \ref{remint} for a better bound.

%_________________________________________________________________
\subsection{Isoperimetric profile with iterated logarithmic decay}\label{ssitlog}

There are explicit groups where the isoperimetric profile decays with a power of iterated logarithms. 
Example of such groups are iterated wreath products $F \wr (F\wr N)$ where $F$ is finite and $N$ is a nilpotent group whose growth is polynomial of degree $d$.
For such groups, one has $\dfrac{C_1}{\left(\log \log (n)\right)^{1/d}} \leq \Lambda_G(n) \leq \dfrac{C_2}{\left(\log \log (n)\right)^{1/d}}$. 
Iterating further the wreath products (with finite groups) gives a profile with more iterated logarithms; see Erschler \cite[Theorem 1]{Ers} or Gromov \cite[\S{}8.1]{Gro}.

Let $\log^{(k)}(x) \vcentcolon = \underbrace{\log \log \cdots \log}_{k \text{ times}} (x)$ and $\exp^{(j)}(x) \vcentcolon = \underbrace{\exp \exp \cdots \exp}_{j \text{ times}} (x)$.

\begin{prop}\label{tborninfloglog-p}
Let $C_1, C_2, \alpha,\beta > 0$ and let $k\geq2$ be a positive integer. Let $G$ be an infinite connected amenable graph with partial self-isomorphisms such that
\[\dfrac{C_1}{\left(\log^{(k)}(n)\right)^\alpha} 
\leq 
\Lambda_G(n) 
\leq 
\dfrac{C_2}{\left(\log^{(k)}(n)\right)^\beta} 
\]

Then there exists some positive constants $K$ and $C$ such that the following inequality holds for infinitely many $N's$:
\[
\boxed{
\dfrac{\Sep(N)}{N} \geq K \dfrac{\Lambda(N)}{\exp^{(k-1)}\left(C\left(\log^{(k)}(N)\right)^{\frac{\alpha}{\beta}}\right)}
}
\]
\end{prop}

\begin{remk}
When $k=2$ and $\alpha = \beta$, the denominator of the right-hand-side is a power of $\log(n)$.
\end{remk}

\begin{remk}\label{not_trivial} Note that the right-left side expression is at least equal to $\frac{1}{N^\eps}$ if $N$ is large enough:
\begin{itemize}
\item First, note that for any $\eps>0$, $\Lambda(N) \geq {N^{-\frac{\eps}{2}}}$ for $N$ large enough. 
\item Second, let $\eps > 0$. 
For $N$ large enough the following inequalities hold:
\[\log^{(k)}(N)
\leq 
\left( \frac{\log^{(k-1)}(N)}{2C}\right)^{\frac{\beta}{\alpha}}
\leq
\left( \frac{1}{C}\right)^{\frac{\beta}{\alpha}} \left(\log^{(k-1)}(N^{\frac{\eps}{2}})\right)^{\frac{\beta}{\alpha}}
\]

\end{itemize}

Therefore we have:
\[
C\left(\log^{(k)}(N)\right)^{\frac{\alpha}{\beta}} \leq
\log^{(k-1)}\left(N^{\frac{\eps}{2}}\right)\]
\emph{i.e.}:
\[\exp^{(k-1)}\left(C\left(\log^{(k)}(N)\right)^{\frac{\alpha}{\beta}}\right)
\leq
N^{\frac{\eps}{2}}\]
\end{remk}

\begin{proof}
Let $\eps \in (0,1)$ and let $n$ be a positive integer.
Write $C 
= 
\left( \dfrac{C_2}{C_1 (1 - \eps)}\right)^{\alpha/\beta}
$
and $m$ the smallest integer such that 
$\log^{(k)}(m) 
\geq 
C\left(\log^{(k)}(n)\right)^{\frac{\alpha}{\beta}}$. Then we have $\Lambda(mn)\leq\Lambda(m)\leq \left(1-\eps \right) \Lambda(n)$.

Then according to Theorem \ref{tpropsep-t}, there is a $N\in [n,2mn]$ such that
\[
\dfrac{\Sep(N)}{N} 
\geq 
\eps  \dfrac{\Lambda(N)}{4\log(m) + 8}
\geq \frac{\eps}{5} \dfrac{\Lambda(N)}{\exp^{(k-1)}\left(C\left(\log^{(k)}(N)\right)^{\frac{\alpha}{\beta}}\right)}
\]
if $n$ is large enough. 

Choosing any $\eps\in\left(0,1\right)$, we are done with $K=\frac{\eps }{5}$.
\end{proof}

Note that, even when $\alpha = \beta$ the lower bound is significantly weaker than $\Lambda$ (because a power of $\log(N)$ is significantly bigger than $\log \log (N)$). 
This tendency continues for isoperimetric profiles which are even closer to being constant. 
For example, if $\Lambda \simeq 1/ \log \log \log (n)$, then the bound on $\Sep(N)/N$ is of the form $\Lambda(N) / \log(N)^{\log^{\eta}(\log(N))}$, with $\eta=1-C$.

\begin{cor}
Assume $G = F \wr (F \wr N)$ where $F$ is a finite group and $N$ is a nilpotent group whose growth is polynomial of degree $d$. 
Then there are constants $C > 1$, $K_1$ and $K_2 >0$ so that for any $n \geq K_1$, there is a $N \in [n,2n\log(n)^C]$ with
\[
\dfrac{\Sep(N)}{N} 
\geq 
K_2 \dfrac{\Lambda(N)}{\log(N)^C}
\simeq 
\dfrac{1}{\log(\log(N)) \cdot \log(N)^C}
\]
\end{cor}

It is hard to tell if this lower bound is sharp.
Again let us note that in a overly optimistic scenario, where all $N$ were optimal and $\Lambda(N) = \frac{1}{\log (\log(N))}$ then Lemma \ref{tbase-l} would give $\frac{\Sep(N)}{N} \succcurlyeq \frac{1}{\log^2(\log(N)) \log(N)}$ (for $N$ large enough). 
Hence in this case, the loss in sharpness could have happened at any stage. 
Note however, that some power of $\log$ need to be present as we exhibit an upper bound which also decays as a power of $\log$, see Remark \ref{remlog}.

\subsection{A qualitative approach}\label{single_bound}

In this subsection, we focus on applications of Theorem \ref{tpropnonsym-t} from a single upper bound on the isoperimetric profile. 
This gives an improvement on the lower bounds obtained on the separation profile, but this comes to the cost of a weaker control on the frequency of the integers for which the bound holds. 
Inspired by the formulation of Theorem \ref{tpropnonsym-t}, we quantify the decreasing of real functions in the following way:
\begin{defn}
	Let $ f\colon \mathbb{R}_{>0}\rightarrow\mathbb{R}_{>0} $ be a continuous non-increasing function such that$\displaystyle \lim_{x \to \infty} f(x)  = 0$.
	For any $ \delta \in \left(0,1\right) $, we define the \emph{$ \delta -$geometric decay function} of $ f $ as:
	\[ p_{f}^{\delta}(x) \vcentcolon= f^{-1}\left(\delta f(x)\right) = \min \left\{x' \mid f\left(x'\right) \leq \delta f(x) \right\}   \]
	We define the $ \delta- $geometric decay function of a function from $ \mathbb{N}^*$ to $\mathbb{R}_{>0} $  as the $ \delta- $geometric decay function of either a natural extension, either a piecewise affine extension.
\end{defn}
We can state the following theorem:
\begin{thm}\label{tpropnonsym-abs}
	Let $G$ be an infinite connected amenable graph of bounded degrees. Let $ g $ be a continuous non-increasing positive function such that:
	\begin{itemize}
		\item $\displaystyle \lim_{n \to \infty} g(n) = 0$
		\item for any large enough $ n $, $ \Lambda(n) \leq g(n) $
	\end{itemize}
	Then for any $n\geq 1$ there exists an integer $N\in \left[n,p_{g}^{1/8}\left(n\right)\right]$ such that 
	$$
	\boxed{
		\dfrac{\Sep_{G}(N)}{N} \geq \dfrac{1}{8}\dfrac{\Lambda_{G}(n)}{ \log\left(\frac{p^{1/8}_{g}(n)}{n}\right) + 1}.
	}
	$$
\end{thm}
\begin{proof}
	This follows from the two lemmas below.
\end{proof}
\begin{lem}\label{lem_dec}
	Let $G$ be an infinite connected amenable graph of bounded degrees. Then for any $n\geq 1$ there exists an integer $N\in \left[n,p^{1/4}_{\Lambda_G}(n)\right]$ such that 
	$$
	\boxed{
		\dfrac{\Sep_{G}(N)}{N} \geq \dfrac{1}{8}\dfrac{\Lambda_{G}(n)}{ \log\left(\frac{p^{1/4}_{\Lambda_G}(n)}{n}\right) + 1}.
	}
	$$
\end{lem}
\begin{proof}
	This is straightforward using Theorem \ref{tpropnonsym-t}, taking $\eps = \tfrac{1}{2}$, $ m = \left\lfloor p^{1/2}_{\Lambda_G}(n) \right\rfloor \in \left[n,p^{1/4}_{\Lambda_G}(n)\right]$ and $ p(m) =  \left\lfloor p^{1/4}_{\Lambda_G}(n) \right\rfloor$ . Note that in the degenerated cases where $ m = p(m) $ the integer $ p(m) $ is optimal since $ \frac{\Lambda(p(m))}{\Lambda(p(m)-1)} \leq 1/2 $.
\end{proof}
\begin{lem}\label{comp_decay}
	Let $ f, g\colon \mathbb{R}_{>0}\rightarrow\mathbb{R}_{>0} $ two continuous non-increasing functions such that $\displaystyle \lim_{x \to \infty} f(x)  = \lim_{x \to \infty} g(x) = 0$.
	
	We assume that for any $ x>0 $ we have $ f(x) \leq g(x) $. Then there exists infinitely many positive integers $ \left(n_i\right)_{i\geq0} $ such that for any $ i$:
	\[ p_{f}^{1/4}\left(n_i\right) \leq p_{g}^{1/8}\left(n_i\right) \]	
\end{lem}
\begin{proof}
	The proof is follows from the pigeonhole principle. By contradiction, we assume that there exists an integer $ N $ such that for any $ n\geq N $ we have $ p_{f}^{1/4}\left(n\right) > p_{g}^{1/8}\left(n\right) $. Let $ m $ be such that $ 2^{m}> \frac{g(N)}{f(N)} $. We compose $ m $ times the $ 1/4- $geometric decay function of $ f $:
	\begin{align*}
	f\left ({p_{f}^{1/4}}^{\circ m}\left(N\right)\right )
	&= \frac{1}{4^{m}} f(N)\\
	&> \frac{1}{8^{m}} g(N)\quad\text{by definition of $ m $}\\
	&= g\left ({p_{g}^{1/8}}^{\circ m}\left(N\right)\right )\\
	&\geq g\left ({p_{f}^{1/4}}^{\circ m}\left(N\right)\right )
	\end{align*}
	This is a contradiction. The last inequality comes from the fact that $p_{g}^{1/8}$ is non-decreasing and then for any $ n\geq 1 $ the function $ g\circ {p_{g}^{1/8}}^{\circ (n-1)} $ is non-increasing, and moreover by assumption ${p_{f}^{1/4}}^{\circ (m-n+1)} \left(N\right) = p_{f}^{1/4} \left({p_{f}^{1/4}}^{\circ (m-n)}\left(N\right)\right) > {p_{g}^{1/8}} \left({p_{f}^{1/4}}^{\circ (m-n)} \left(N\right)\right)   $. Therefore for any $ n\geq1 $: $$g\circ {p_{g}^{1/8}}^{\circ (n)} \circ {p_{f}^{1/4}}^{\circ (m-n)} \left(N\right)\geq  g\circ {p_{g}^{1/8}}^{\circ (n-1)} \circ {p_{f}^{1/4}}^{\circ (m-n+1)} \left(N\right)   $$ We conclude by induction on $ n $.	
\end{proof}
We can deduce the following theorem:
\begin{thm}\label{thm_gen_res}
	Let $ G $ be a finitely generated solvable group. If there exists $ \epsilon \in \left(0,1\right) $ and $ c > 0 $ such that for any large enough integer $ n $ we have:
	\[ \Sep_{G}(n) \leq c n^{1 - \epsilon} \]
	Then $ G $ is virtually nilpotent.	
\end{thm}
It is known that any nilpotent group of rank $ d $ has a separation profile equivalent to $ n^{\frac{d-1}{d}} $ (see \cite[Theorem 7]{HMT}). 
We show here that, among solvable groups, the separation profile is able to reveal nilpotence. 
This dichotomy is quite sharp since, according to Corollary \ref{tcomp-c}, the separation profile of the classical lamplighter group $ \mathbb{Z}_{2} \wr \mathbb{Z} $ (as well as any polycyclic group) is bounded above by $ \frac{n}{\log^{1 - \eps}(n)}$ for any $\eps >0$. Note that this dichotomy is definitively not true in general as non-amenable groups give counterexamples: for any $ d\geq 3 $,  the $d$-dimensional hyperbolic space has a separation profile equivalent to $ n^{\frac{d-2}{d-1}} $ (see \cite[Proposition 4.1.]{BST}) or, more spectacularly, free groups has a bounded separation profile. 

This theorem partially answers a question posed to us independently by David Hume and Jérémie Brieussel:
\begin{ques}\label{qu_res}
	Is there an exponential growth solvable group $ \Gamma $ such that $ \Sep_{\Gamma}(n) \prec \frac{n}{\log(n)} $ ?
\end{ques}
See Question \ref{qres2} in \S{}\ref{sques} for further discussions of this topic.

\begin{proof}[Proof of Theorem \ref{thm_gen_res}]
Recall that for any group $\Delta$ we define the the derived series of $ \Delta $ as the sequence of groups $ \left(\Delta^{(i)}\right)_{i\geq 0} $ defined inductively by $\Delta^{(0)} = \Delta$, $\Delta^{(i)} = \left[\Delta^{(i-1)}, \Delta^{(i-1)}\right]$. A group is solvable if and only if $ \Delta^{(i)} = \left\{e\right\} $ for some $ i $ and the smallest such $ i $ is called the derived length of $ \Delta $.

Let $ r $ be the size of a finite generating set of $ G $ and $ d $ be the derived length of $ G $. If $ G $ is an abelian group, then the conclusion of Theorem \ref{thm_gen_res} is valid, then we can assume that $ d $ is at least equal to 2.  Let $ \mathbb{F}_r $ be ``the'' free group on $ r $ generators, labelled by a generating set of $ G $ of size $ r $, and let $\textbf{S}_{d,r} \vcentcolon = \mathbb{F}_r / \mathbb{F}_r^{(d)} $ be the free solvable group $ r- $generated of length $ d $. $ G $ is a quotient of $\textbf{S}_{d,r}$, considering the well-defined surjective group homomorphism $\pi_{G} \colon \textbf{S}_{d,r} \rightarrow G$.
From Tessera \cite[Proposition 5.5]{Tes13}, we have \[
 \Lambda_{G} \leq \Lambda_{\textbf{S}_{d,r}} \]

Additionally, L. Saloff-Coste and T. Zheng explicited in the introduction of \cite{SF_Z_isop} the isoperimetric profile of the free solvable groups, namely:\[
\Lambda_{\textbf{S}_{d,r}}(n) \simeq  \left(\frac{\log^{(d)}(n)}{\log^{(d-1)}(n)}\right)^{1/r}  \]
	
	Combining those two inequalities (and the fact that $\log^{(d)}(n) \leq \log^{(d-1)}(n) ^{1/2}$ for $n \geq 4^{d-1}$), we get that there exists some constants $ c,d,r $ such that for any large enough $ n $, we have 
	\[
	\stepcounter{equation}\tag{\theequation}\label{isop_resol} 
	\Lambda_{G}(n) \leq \frac{c}{\left(\log^{(d-1)}(n)\right)^{1/2r}} 
	\]
	\begin{fac}\label{free_decay}
		Let $g(n) = \frac{c}{\left(\log^{(d-1)}(n)\right)^{1/2r}}$.
		Then for any large enough $ n $ we have  
		\[ 
		p_{g}^{1/8}\left(n\right) \leq \exp^{(d-1)}\left(8^{2r} \log^{(d-1)}(n)\right) 
		\]
	\end{fac}
	\begin{proof}
	Let $ m,n $  be such that $ m \geq \exp^{(d-1)}\left(8^{2r} \log^{(d-1)}(n)\right)$. Let us write $ m' = \log^{(d-1)}(m) $ and $ n' = \log^{(d-1)}(n) $. We have $ m' \geq 8^{2r} n' $. Hence, if $ n $ is larger than some constant (depending on $r$ and $d$), we have $ g(m) \leq \frac{1}{8} g(n) $, which is the required inequality.
	\end{proof}
	Let's assume that $ G $ has an exponential volume growth. From Theorem 1 of \cite{CS}, there exists a positive constant $ c_l $ such that the following inequality is true for any $ n $:\[ \stepcounter{equation}\tag{\theequation}\label{low_bnd} \Lambda_{G}(n) \geq \frac{c_ l}{\log(n)}
	\]
	Then combining Theorem \ref{tpropnonsym-abs} (or more precisely Lemma \ref{lem_dec}) with Fact \ref{free_decay}, \eqref{isop_resol} and \eqref{low_bnd}, we get that there exists infinitely many integers $ \left(n_i\right)_{i\geq 0} $ such that for any $ i $:
	\[ 
	\frac{\Sep_{G}\left (n_i\right )}{n_i} \geq \frac{c_l}{\log(n_i)}\times \frac{1/16}{\exp^{(d-2)}\left(8^{2r} \log^{(d-1)}(n_i)\right)} 
	\]
	Let $ \epsilon \in \left(0,1\right) $. Following the reasoning of Remark \ref{not_trivial}, we can deduce that the right-hand side is at least equal to $ \frac{1}{{n_i}^{\epsilon}}$ if $ n_i $ is large enough, which can be assumed without any loss of generality. 
	
	We have shown that if the group $ G $ has exponential volume growth, then its separation profile dominates along a subsequence every sublinear power function. By contraposition, if the assumptions of Theorem \ref{thm_gen_res} are satisfied, that means that the group $ G $ does not have exponential growth, meaning that it is virtually nilpotent according to the usual dichotomy for solvable groups (see for example \cite{dlH}).	
\end{proof}

\begin{remk}
The proof of Theorem \ref{thm_gen_res} shows that for a free metabelian group of rank $r$ (\ie free solvable group of derived length 2) one has $\dfrac{\Sep(n)}{n} \geq \dfrac{K}{\log(n)^{1+\tfrac{1}{r}+\eps}}$ for any $\eps >0$ and infinitely many $n$.

Indeed, for any $\eps >0$, $\dfrac{ \log \log n}{\log n} > \dfrac{1}{(\log n)^{1+\eps}}$ for $n$ large enough.
Note that this is better than the bound one would obtain through Proposition \ref{tborninflog-p}.
\end{remk}

\begin{remk}\label{remint}
The estimates of the proof of Theorem \ref{thm_gen_res} (that is combining Lemma \ref{lem_dec} with Fact \ref{free_decay} and the lower bound on isoperimetry) show that for a group of intermediate growth (\ie with $e^{n^a} \preccurlyeq |B_n| \preccurlyeq e^{n^b}$) one has $\dfrac{\Sep(n)}{n} \geq \dfrac{K}{\log(n)^{1+1/a}}$ for infinitely many $n$. This is better than the bound from Corollary \ref{tborninfint-c}.
\end{remk}

\subsection{Limitations of Theorem \ref{tpropsep-t}}\label{sslimit}

Now we will be interested in graphs where the conclusion of Theorem \ref{tpropsep-t} becomes trivial.

A first source of loss of sharpness is the uncertainty on the isoperimetric profile. This happens for example when the  profile is bounded by different numbers of iterated logarithms. However, we will consider this limitation due to lack of information as superficial. The approach of \S \ref{single_bound} supports this point of view.

A second source, much more deeper to our opinion, is the decay of the isoperimetric profile. As we noticed before, Theorem \ref{tpropsep-t} gives nothing in the case of graphs with almost constant isoperimetric profiles, since the integer $m$ doesn't exist for $n$ large enough. We will see in this subsection an example of isoperimetric profile for which the conclusion of the Theorem \ref{tpropsep-t} is trivial.

Erschler \cite{Ers2} and Brieussel \& Zheng \cite{BZ} give an explicit construction of groups with up-to-constant prescribed\footnote{In those examples, isoperimetric profiles aren't prescribed exactly, but up to constants. However, using Theorem \ref{tpropnonsym-abs} we can have the same bounds (up to constants) on infinitely many $ N$'s as using Theorem \ref{tpropnonsym-t} if the isoperimetric profile was exactly prescribed.} isoperimetric profile. We can deduce that the examples below has instances in Cayley graphs, which leaves little hope to generalise Theorem \ref{thm_gen_res} to amenable groups.

First, note that we have no particular control on $N$. Therefore we can consider that the conclusion of Theorem \ref{tpropsep-t} is trivial if for any $\eps>0$, $\frac{n\Lambda(n)}{\log(m_n)} $ is bounded, with $m_n = 
\min\left \{ m : \Lambda(mn) \leq (1-\eps) \cdot \Lambda(n)\right \}$. Indeed, we would only be able to conclude that $\Sep(n)$ - which is non-decreasing - is greater than a bounded function.

This condition is equivalent to the following:
\[\tag{*}\label{condvide1}\forall \eps>0\ \exists \beta>0\ \forall n\gg 1\ \Lambda\left(n\exp\left(\beta n \Lambda(n)\right)\right) \geq (1-\eps) \cdot \Lambda(n)\]

\paragraph{A first example}

Let's define a sequence $(a_{k})_{k\geq 0}$ inductively: $a_0=1$ and $\forall k\geq 0\ a_{k+1}=a_k \exp(a_k)$.

Now we can define a unique piecewise affine function $f:\R_{\geq1}\rightarrow \R$ such that $f(a_k)=\frac {1} {k+1}$ for any $k$.

For any integer $k$ we denote by $I_k$ the interval $\left[ a_k, a_{k+1}\right]$. Let $k$ be an integer. Let $x$ be an element of $I_k$.

Then we have:
\begin{align*}
f\Big(x \exp \big( f(x) x \big) \Big) 
&\geq
f\big(x\exp(x) \big)\\
&\geq
f\left(a_{k+2}\right)\\
&\geq
\frac{k}{k+2} f(a_k)\\
&\geq
\frac{k}{k+2} f(x)
\end{align*}

Now let $x$ is at most $a_k$, then $x$ is in an interval $I_{k'}$ with $k'$ beng at most equal to $k$. Then we can deduce that for any $x\geq a_k$ we have $f\left(x \exp\left(f\left(x\right)x\right)\right) \geq
\frac{k}{k+2} f(x)$. % ZZZ beng?

From this fact we can deduce that \eqref{condvide1} holds with $\beta = 1$ and $n\geq a_k$, where $k$ is the smallest integer such that $\frac{k}{k+2}\geq 1-\eps$

\paragraph{A second example}

 A subtler counter-example is given by a variation of the previous example, with an exponential decay. Indeed, we will see that if $\eps=1/2$, there exists a $\beta$ such that \eqref{condvide1} holds, while for any $\eps<\frac{1}{2}$, \eqref{condvide1} doesn't hold for any $\beta$.

Let's consider the sequence $(a_{k})_{k\geq 0}$ defined previously, and let's define $f$ recursively in the following way:
\begin{itemize}
\item $f(a_0)=1$, $f(a_1)=1/2$
\item $f$ is affine between $a_0$ and $a_1$
\item $f(x\exp(x))= \frac{1}{2}f(x)$ for any $x\geq1$.
\end{itemize}

By construction, \eqref{condvide1} holds for $\eps = \frac{1}{2}$.

Let's show that for any $\eps<1/2$ we have $\forall \beta\in(0,1)\ f\left(n\exp\left(\beta f(n) n\right)\right) \leq (1-\eps) f(n)$, assuming $n$ is large enough.

Note that for any $\alpha>0$ we have $f\left(\beta f(n) n\right) \leq (1+\alpha)f(n)$ for $n$ large enough (this comes from the mean value inequality and the fact that $\lim_{n\rightarrow\infty} \frac{a_k}{2^k} = 0$). From this fact we can deduce:

\begin{align*}
f \left(n\exp\left(\beta f(n) n\right)\right)
&\leq
f\left(\beta f(n) n\exp\left(\beta f(n) n\right)\right)\\
&=
\frac{1}{2} f\left(\beta f(n) n\right)\\
&\leq
\frac{1+\alpha}{2} f\left(n\right)
\end{align*}

%____________________________________________________________
% ____________________________________________________________
\section{Upper bounds on the separation profile}\label{suppe}
% ____________________________________________________________

\subsection{From growth}\label{ssgrow}
%_______________________

The aim of this subsection is to obtain upper bound from on the separation profile using the growth of balls in the graphs.
Let $d$ denote the combinatorial distance in the graph, then $B_n(x) = \{ v \in VG \mid d(x,v) \leq n\}$ is the ball of radius $n$ with centre $x$.

In order to effectively apply this method, the upper bound on the size $B_n(x)$ should be independent of the choice of the ball's centre $x$. 

\begin{lem}\label{tbornboul-l}
Assume $G$ is a graph such that $\displaystyle \sup_{x\in VG} |B_n(x)| \leq e^{f(n)}$ and $\frac{f(n)}{n} \to 0$. For any subgraph $G'$ let $\beta_n(x)$ be the cardinality of a ball [inside the subgraph] of radius $n$ centred at $x$. Let $n_0$ be such that $\displaystyle \sup_{n \geq n_0} \frac{f(n)}{n} \leq 1$. Then for any $n \geq n_0$ and $x\in VG$ there is a $\ell \in [n,2n]$ such that
\[
 \dfrac{ \beta_{\ell+1}(x) - \beta_\ell(x)}{\beta_\ell(x)} \leq \dfrac{2f(n)}{n}
\]
\end{lem}
\begin{proof}
Let us alleviate notation by using $\beta_j \vcentcolon = \beta_j(x)$ and $\sigma_n = \beta_n - \beta_{n-1}$. Let $C_{n,k} = \displaystyle\min_{i \in [n, n+k[} \frac{ \beta_{i+1}}{\beta_i}$. Then $\beta_{n+k} \geq \left (C_{n,k}\right )^k \beta_n$. However, since the growth of the extra $k$ steps is bounded by $e^{f(k)}$, $\beta_{n+k} \leq \beta_{n-1} + \sigma_n e^{f(k)} = \beta_n + \sigma_n (e^{f(k)} -1)$. Thus
\[
\left (C_{n,k}\right )^k \leq 1 + \frac{ \sigma_n}{\beta_n} (e^{f(k)} -1) \leq e^{f(k)}.
\]
This implies that $C_{n,k} -1 \leq e^{f(k)/k} -1$. If $\frac{f(k)}{k} \leq 1$ then $C_{n,k} -1 \leq e^{f(k)/k} -1 \leq (e-1)\frac{f(k)}{k} \leq 2 \frac{f(k)}{k}$. Taking $k = n$ yields the conclusion.
\end{proof}

\begin{prop} \label{tbornsupsep-p}
Let $G$ be a graph so that $\displaystyle \sup_{x\in VG} |B_n(x)| \leq e^{f(n)}$ for a function $f$ with $\displaystyle \lim_{n \to \infty} \frac{f(n)}{n} = 0$. Assume the degree of the vertices is bounded by $d$. Then there is a constant $K$ (depending on $f$) such that for any integer $N >K$,
\[
\frac{\Sep(N)}{N} \leq 4d \frac{f( \frac{f^{-1}(\ln \tfrac{N}{2})-1}{2})}{f^{-1}(\ln \tfrac{N}{2})-1}
\]
\end{prop}
\begin{proof}
For any subset $F \subset VG$ of cardinality $N$, consider the balls $B'_n(x)$ of radius $n$ in the subgraph induced by $F$. 
(It does not matter where the centre $x$ of the ball is, if one could choose, it would probably be best to choose a point that realises the diameter of $F$).
Note that $|B'_n(x)| \leq |B_n(x)| \leq e^{f(n)}$.
Let $n_0$ be the largest integer such that $e^{f(n_0)} \leq N/2$.
Applying Lemma \ref{tbornboul-l} with $n = \lfloor n_0/2 \rfloor$ implies that for some $k \in [\tfrac{n_0-1}{2},n_0]$, 
\[
\frac{|\del B'_k|}{|B'_k|} \leq d \frac{ |B'_{k+1}| - |B'_k|}{|B'_k|} \leq d\frac{2f(n)}{n}
\]

Consequently, the Cheeger constant of $F$ is at most $2d\frac{f(n)}{n}$.
This shows that $\frac{\Sep(N)}{N} \leq 2d\frac{f(n)}{n}$, as long as $\frac{f(n)}{n} \leq 1$.
The constant $K$ is $e^{f(n_1)}$, where $n_1$ is the smallest integer so that $\frac{f(n)}{n} \leq 1$ for any $n \geq n_1$.
\end{proof}

\begin{cor}\label{tbornsupint-c}
Let $G$ be a graph so that $\displaystyle \sup_{x\in VG} |B_n(x)| \leq K_1 e^{K_2 n^\alpha}$ for some constants $K_1,K_2 >0$ and $\alpha \in [0,1)$.
Then, there are constants $L_1,L_2 >0$ so that, for any $N >L_1$ large enough,
\[
\frac{\Sep(N)}{N} \leq \frac{L_2}{(\ln N)^{\frac{1}{\alpha}-1}}
\]
\end{cor}
\begin{proof}
Use Proposition \ref{tbornsupsep-p} with $f(x) = K_2 x^\alpha + \log K_1$.
\end{proof}

\begin{cor}\label{tbornsuppol-c}
Let $G$ be a graph so that $\displaystyle \sup_{x\in VG} |B_n(x)| \leq K n^d$ for some constants $K,d >0$.
Then, there are constants $L_1,L_2 >0$ so that, for any $N >L_1$ large enough,
\[
\frac{\Sep(N)}{N} \leq \frac{L_2 \log (N)}{N^{1/d}}
\]
\end{cor}
\begin{proof}
Use Proposition \ref{tbornsupsep-p} with $f(x) \simeq d \log x + \log K$.
\end{proof}

\subsection{From compression}\label{sscomp}
%____________________________

Another upper bound on the separation profile can be obtained if the groups has a good embedding in $L^p$-spaces. Recall that for a finite graph $X$ and a 1-Lipschitz embedding $F:X \hookrightarrow  Y$ in a $L^p$-space, the distortion of $F$ is 
\[
 \textrm{dist}(F) = \max_{x,y \in X, x \neq y} \frac{d(x,y)}{\|F(x) - F(y)\|_Y}
\]
where $d$ is the combinatorial distance on $X$. The $L^p$-distortion of $X$ is $c_p(X) \vcentcolon = \inf \{ \textrm{dist}(F) | F: X \hookrightarrow Y\}$.

P.-N. Jolissaint and Valette \cite[Theorem 1.1]{PNJV} (combined with \cite[Proposition 3.3]{PNJV} as well as an estimate on the $p$-spectral gap from Amghibech \cite{Amg} ) give the following lower bound on the $L^p$ distortion for a finite graph $X$ with $n$ vertices and Cheeger constant $h$:
\[\tag{\dag\dag}\label{eqJoVa}
 c_p(X) \geq K \log(n) h 
\]
where the constant $K$ depends only on $p$ and the maximal degree.

Distortion can also be studied for infinite graphs, but we will here rely on the notion of compression. If $G$ is a [connected] infinite graph, then a compression function for a 1-Lipschitz embedding $\Phi:G \hookrightarrow L^p$ is a function $\rho$ so that
\[\tag{\ddag\ddag} \label{ecomp}
\forall x,y \in G, \qquad \rho( d(x,y) )  \leq \|\Phi(x) - \Phi(y)\|_{L^p}
\]
The compression exponent of $F$ is $\alpha(F) = \displaystyle \liminf_{x\to \infty} \frac{\log \rho(x)}{\log x}$ and the compression exponent of the graph is $\alpha(G) = \sup_{F} \alpha(F)$.

\begin{prop}\label{tcomp-p}
Let $G$ be a connected graph of bounded degree which admits an embedding in a $L^p$-space (for some $p$) with compression function $\rho(x) = k_1+k_2x^a$. Then there is a constant $K$ so that 
\[
 \frac{\Sep(n)}{n} \leq \frac{K}{(\log n)^{a/(2-a)}}
\]
\end{prop}
\begin{proof}
Assume $X$ is a finite subgraph of $G$ of cardinality $n$ with Cheeger constant $h$, maximal degree $k$ and diameter $\delta$. 
Note that the diameter of $X$ is bounded using $n \geq (1+\frac{h}{k})^\delta$, that is $\delta \leq \frac{\log n}{\log(1+\frac{h}{k})}$.

Since $G$ admits a 1-Lipschitz embedding $\Phi$, this embedding restricts to $X$ and \eqref{ecomp} can be rewritten as 
\[
\forall x,y \in X,\qquad  \frac{d(x,y)}{\rho( d(x,y) )}  \geq \frac{d(x,y)}{\|\Phi(x) - \Phi(y)\|_{L^p} }.
\]
The left-hand-side gets only bigger if one looks at $d(x,y) = \delta$. By taking the maximum on the right hand side, this leads to 
\[
 \frac{\delta}{\rho(\delta)} \geq c_p(X).
\]
Using the bound \eqref{eqJoVa} of P.-N. Jolissaint and Valette \cite{PNJV}, this gives 
\[
 \frac{\delta}{\rho(\delta)}  \geq K h \log(n).
\]
For $n$ large enough, $\delta$ is also large so that $\frac{\delta}{\rho(\delta)} \leq k_3 \delta^{1-a}$ for some constant $k_3$. Next using the bound on $\delta$ above:
\[
K h \log(n) \leq k_3 \delta^{1-a} \leq \frac{(\log n)^{1-a}}{(\log(1+\frac{h}{k}))^{1-a}}.
\]
This can be rewritten as
\[
K h (\log(1+\frac{h}{k}))^{1-a} \leq k_3 (\log n)^{-a}
\]
Since $h\leq k$, $\log(1+\frac{h}{k}) \geq \frac{h}{k \log 2}$. With new constants, the inequality reads:
\[
 h^{2-a} \leq K' (\log n)^{-a}
\]
This means that any [connected] subset $X$ of cardinality $n$ inside $G$ has a Cheeger constant of at most $K' (\log n)^{-a/(2-a)}$. From the definition of $\Sep(n)$ it follows that $\Sep(n) \leq n (\log n)^{-a/(2-a)}$. 
\end{proof}

\begin{cor}\label{tcomp-c}
Assume $G$ is a graph with bounded degree and compression exponent $\alpha$ (in some $L_p$-space). Then for any $c < \frac{\alpha}{2-\alpha}$ there is a constant $K$ so that 
\[
\frac{ \Sep(n)}{n} \leq \frac{K}{(\log n)^{c}}
\]

\end{cor}
Here is a [non-exhaustive] list of Cayley graphs for which the compression exponent is known (references below). This table does not always use the case $p=2$; in fact taking $p \to 1$ or $p \to \infty$ often gives better bounds, see Naor \& Peres \cite[Lemma 2.1]{NPspeed}. 

\newcommand{\fn}[1]{\footnotemark[#1]}
\renewcommand{\thefootnote}{(\alph{footnote})}
\newcounter{toto}
\newcommand{\ft}[1]{\setcounter{toto}{#1}(\alph{toto})}
\noindent
\begin{tabular}{|c|c|l|}
\hline
$\alpha(G)$ & $\frac{\alpha}{2-\alpha}$ & Groups\\
\hline
1 & 1 & polycyclic groups\fn{1}, the lamplighter group over $\Z$ with finite lamps\fn{1}, \\
  &   & hyperbolic groups\fn{2}, Baumslag-Solitar groups\fn{3}, 3-manifolds groups\fn{4} \\
\hline
$\to 1$ & 1 & 
lamplighter group over $\Z$ with lamps in $\Z$\fn{5}, lamplighter over $H$ of \\
	&	& polynomial growth with finite lamps or lamps in $\Z$\fn{6} \\
\hline
$\tfrac{1}{2}$ & $\tfrac{1}{3}$ & lamplighter over $\Z^2$ with lamp group $H$ having $\alpha_2(H) \geq \tfrac{1}{2}$\fn{7},\\ 
	&   &  Thompson's group $F$\fn{8}\\
\hline
$\tfrac{1}{2-2^{1-k}}$ & $\tfrac{1}{3-2^{2-k}}$ & iterated wreath products of $\Z$: $(\ldots ((\Z \wr \Z) \wr \Z) \ldots) \wr \Z$  (with $k$ ``$\Z$'')\fn{9} \\
\hline
$\geq \frac{1-\gamma}{1+\gamma}$ & $\geq \frac{1-\gamma}{1+3\gamma}$ & groups with return probability after $n$ steps of a SRW $ \leq K_2e^{-K_1n^\gamma}$\fn{10}\\
\hline
$\geq 1-\nu$ & $\geq \frac{1-\nu}{1+\nu}$ & groups of intermediate growth with $b_n \leq e^{n^\nu}$\fn{11}\\
\hline
$\to \geq \frac{1}{d-1}$ & $\geq \frac{1}{2d-3}$ & free solvable groups $S_{r,d}$ of length $d$ when $d>1$\fn{12}\\
\hline
\end{tabular}

Table's references:\\
\ft{1} Tessera \cite[Theorems 9 and 10]{Tes11}\\
\ft{2} Bonk \& Schramm \cite{BoS} and Buyalo \& Schroeder \cite{BuS} \\
\ft{3} Jolissaint \& Pillon \cite[Corollary 2]{JP}\\
\ft{4} Hume \cite[Theorem 5.4]{H-comp}\\
\ft{5} Naor \& Peres \cite[Lemma 7.8]{NPspeed} and \cite[Theorem 6.1]{NPlamp}; the bound is $\max\{ \frac{p}{2p-1}, \frac{2}{3} \}$, take $p \to 1$ \\
\ft{6} Naor \& Peres \cite[Theorem 3.1]{NPlamp}; the bound is $\max\{\tfrac{1}{p}, \tfrac{1}{2}\}$, take $p \to 1$ \\
\ft{7} Naor \& Peres \cite[Theorem 3.3]{NPspeed}\\
\ft{8} Arzhantseva, Guba \& Sapir \cite[Theorem 1.3]{AGS}\\
\ft{9} Naor \& Peres \cite[Corollary 1.3]{NPspeed}\\
\ft{10} see \cite[Theorem 1.1]{Gy} \\
\ft{11} see either  \cite[Theorem 1.3(b)]{Gy} or Tessera \cite[Proposition 14]{Tes11}. In that case, the bound obtained on $\Sep(n)$ by Proposition \ref{tbornsupsep-p} is better.\\
\ft{12} see Sale \cite[Corollary 4.2]{Sal}; the bound on is $\tfrac{1}{p(d-1)}$ for $p \in [1,2]$, so take $p \to 1$.

\bigskip

There are many other groups for which one can compute the compression (the above list does not exhaust the results in the references). For example, $\alpha(G \times H) = \min \big( \alpha (G) ,\alpha(H) \big)$.
There are also further results: on HNN-extensions see Jollissaint \& Pillon \cite{JP}, on relatively hyperbolic groups see Hume \cite{H-comp}, on wreath products see Li \cite{Li}.

Note that Proposition \ref{tcomp-p} is fairly sharp (in this generality).
Indeed, if one looks at the product of two trees, then the compression exponent is 1. This means Proposition \ref{tcomp-p} ensures, for every $c<1$, the existence of $K_c$ such that $\frac{\Sep(n)}{n} \leq \frac{K_c}{(\log n)^c}$. On the other hand, it was shown by Benjamini, Schramm \& Timár \cite[Theorem 3.5]{BST} that the separation profile of such a space is $\frac{\Sep(n)}{n} \simeq \frac{1}{\log n}$.

\begin{remk}\label{remlog}
The above corollary shows that there are amenable groups for which $\frac{\Sep(n)}{n}$ decreases much more quickly than $\Lambda(n)$. For example, $F \wr (F\wr \Z)$ has $\Lambda(n) \simeq \frac{1}{\log \log n}$. On the other hand this group has an isometric embedding in a Cayley graph of $\Z \wr (\Z \wr \Z)$. In particular, its compression exponent is at least $\tfrac{4}{7}$. This implies that  $\frac{\Sep(n)}{n} \preccurlyeq \frac{1}{(\log n)^c}$ for any $c <\tfrac{2}{5}$.
\end{remk}

\begin{remk}
It is possible to show that if there is an embedding with $\rho(x) \geq K_1 (\log^{(k)} n)^\alpha$ (where $\log^{(k)}$ denotes $k$ iterated logarithms) then the conclusion of Proposition \ref{tcomp-p} is that \\
\centerline{$\dfrac{\Sep(n)}{n} \leq \dfrac{K'}{(\log^{(k+1)} n)^{\alpha/2}}$.}
Compression function of this sort follow from the methods of \cite[\P{} before Remark 3.4]{Gy}. It  can be shown that any [amenable] group where $P^n(e) \leq K_1 \mathrm{exp}(n/\log^{(k)}n)$ has an embedding in some Hilbert space with $\rho(x) \geq (\log^{(k)}n)^{1/2}$. In fact (thanks to Kesten's criterion for amenability), one can get for any amenable group an upper bound on the separation profile.
\end{remk}

%__________________
%__________________
\section{Questions}\label{sques}
%__________________

Although we showed that there are plenty of optimal integers, it turns out it's incredibly hard to describe optimal sets. 
In the case of $\Z^d$ this can probably be achieved with the Loomis-Whitney inequality (see \cite{LW}).
\begin{ques}\label{qheis}
Give an explicit description of the optimal sets in the discrete Heisenberg groups (or in any amenable group which is not virtually Abelian).
\end{ques}
For the ``continuous'' version of the Heisenberg group, this is an old open question.
But perhaps the discrete case is easier.

More generally, one could ask whether it is possible to find the optimal sets in semi-direct products of ``well-known cases'': assuming the optimal sets of the [finitely generated] groups $G_1$ and $G_2$ are known [for some generating sets $S_1$ and $S_2$], can the optimal sets of $G_1 \rtimes G_2$ be of the form $F_1 \times F_2$ (where $F_i$ is an optimal set for the Cayley graph of $G_i$ w.r.t. $S_i$)?

Another interesting question on optimal sets would be the following:
\begin{ques}\label{densepoly}
If $G$ is a graph whose isoperimetric profile is is known up to a multiplicative constant. What can we say about the density of sets whose separation is good?
\end{ques}
Let us shortly describe two interpretations of this question. First, Proposition \ref{tseppolgen-p} only uses the fact that $p(n) \leq K n^c$ for some $K>0$ and $c>1$. 
This gives a fairly low density of optimal integers, leaving open the possibility for much higher densities.
For example, if $K=1$ and $c=2$, then the sequence of optimal integers could be as sparse as $2,4,16,256, \ldots$ 

Second (in the spirit of local separation), one could also fix some $n$, $r$ and $K$ and look at the density of vertices $x$ for which a ball of radius $r$ contains a set of size $n$  which is up to a multiplicative factor of $K$ as hard to cut as the best set for that given $n$.

Here are many inequalities between the separation and isoperimetric profile which seem natural (they might be easy, or hard, to prove or disprove):
\begin{ques}\label{qgroupe}~
\begin{enumerate}
\item If $G$ is the Cayley graph of a group, more generally a vertex-transitive graphs, $\displaystyle \frac{\Sep(N)}{N} \overset{?}\preccurlyeq \Lambda(N) $.
\item If $G$ is the Cayley graph of an amenable group, $\displaystyle \frac{\Sep_G(N)}{N} \overset{?}\succcurlyeq \Lambda_G(N/2) - \Lambda_G(N)$
\item If $G$ is the Cayley graph of a polycyclic group, $\displaystyle \frac{\Sep_G(N)}{N}  \overset{?}\simeq \frac{1}{\log n} $ {\small (For such groups $\Lambda_G(N) \simeq \frac{1}{\log n}$.)}
\item If $G$ is the Cayley graph of a group, is $\displaystyle \frac{\Sep(N)}{N} \overset{?}\succcurlyeq N\big( \Lambda(N-1) - \Lambda(N) \big)$
\end{enumerate}
\end{ques}

The following associated question was also posed to us in connection with Question \ref{qu_res}:
\begin{ques}\label{qres2}
Does the classical lamplighter group $ \mathbb{Z}_{2}\wr\mathbb{Z} $ coarsely embeds in any exponential growth solvable group ? 
\end{ques}
A positive answer to this question would give a (negative) answer to Question \ref{qu_res}. 
In fact, regular maps from the lamplighter to solvable groups (of exponential growth) would be enough (and should be easier to produce).
Note that one cannot replace the lamplighter with a polycyclic group (of exponential growth) in Question \ref{qres2}.
Indeed, the asymptotic dimension increases under a regular map (see Benjamini, Schramm and Timàr \cite[\S{}6]{BST}) and the classical lamplighter has asymptotic dimension 1 while polycyclic groups have dimension $\geq 2$ (they are finitely presented; see Gentimis \cite{Gen} for both results). 
Consequently, there are no regular maps from any polycyclic group to the classical lamplighter group (which is a solvable group). 

We can also ask Questions \ref{qu_res} and \ref{qres2} more generally for exponential growth amenable groups. 
In this larger setting, a positive answer to Question \ref{qres2} should be quite daunting, since it also gives a coarse embedding of a $3$-regular tree (which is an important open question). 
% The existence of regular trees in Cayley graphs of the lamplighter groups dates back at least to Woess \cite{Woess}.

\newpage
	
	\appendix
	\section{Local separation profiles}\label{appendix_local_sep}
	
	In this appendix, we will study a local variant of the separation profile. We found it relevant in two contexts. First, in $\Z^d$ percolation clusters, where considering classical separation profile is trivial: since in almost every percolation configuration one can find arbitrary large balls, the profile is almost surely equal to the separation profile of $\mathbb{Z}^{d}$. Second, it can tackle the issue of the density of high separation subgraphs in non-vertex transitive graphs.
	%Hence it is useful to introduce our more restrictive notion of separation, \emph{local} separation. 
	
	We will first define it, give a local version of Theorem \ref{tpropnonsym-t} (see Theorem \ref{tpropnonsym-t-loc}), some applications to percolation clusters in $\mathbb{Z}^{d}$. Finally, we will give a theorem with a more abstract approach, that applies in graphs of polynomial growth and of isoperimetric dimension larger than one, see Theorems \ref{thm_poly1} and \ref{thm_poly2}.
	
	\begin{defn}
		Let $(G,v)$ be a rooted graph. Let $\rho:\R_{\geq 1}\rightarrow \R_{\geq 1}$ be a non-decreasing function.  We define the $(\rho,v)$-local separation profile as:
		\[\Sep_{G}^{\rho,v}(n) \vcentcolon = \sup_{|F|\leq n\text{ and } F \subset B_G(v,\rho(n))} |F|\cdot h(F)\]
	\end{defn}
	
	In comparison to the classical separation profile, which is defined as $\Sep_G(n) = \sup_{|F|\leq n} |F| \cdot h(F)$, there is an extra condition restricting the 	subgraphs to lie in a given sequence of balls. 
	One can think of it as searching for graph with big separation, but not too far from $x$; the ``not too far''-part is quantified by the function $\rho$. 
	
	As for the classical separation profile, this local variant gives obstructions for the existence or regular maps (see Lemma 1.3 of \cite{BST}). We remind the reader the definition of a regular map:
	
	\begin{defn}\label{dregular}
		Let $ X $ and $ Y $ be two graphs of uniformly bounded degrees. A map $ f\colon X \rightarrow Y $ is said to be \emph{regular} if there exists a constant $ \kappa > 0 $ such that the following two conditions are satisfied:

		\begin{itemize}
			\item $ \forall x_1,x_2 \in X\ d(f(x_1),f(x_2))\leq \kappa \left(d(x_1,x_2) +1 \right)$
			\item $\forall y\in Y\ \left| f^{-1}\left(\{y\}\right)\right|\leq \kappa$
		\end{itemize}		
	\end{defn}

	The local separation profiles satisfies the following monotonicity:
	
	\begin{prop}
		Let $ \left(X,x_0\right) $ and $ \left(X,x_0\right) $ be two rooted graphs of uniformly bounded degrees. Let $\rho:\R_{\geq 1}\rightarrow \R_{\geq 1}$ be a non-decreasing function. Let $ f\colon X \rightarrow Y $ be a regular map such that $ f(x_0) = y_0 $. Then there exists a constant $ K>0 $ such that for any $ n $:
		
		\[ \Sep_{X}^{\rho,v}(n) \leq K\Sep_{Y}^{\rho\left(K \cdot\right),v}(n) \]
		
	\begin{proof}
		The same proof as the proof of Lemma 1.3 of \cite{BST} works.
	\end{proof}

	\end{prop}

	\begin{remk}
		Mind that the constant $ K $ appears both in factor of the separation profile, and in the argument of $\rho$ (we define $\rho\left(K \cdot\right)$ by $n\mapsto\rho\left(K n \right)$).
	\end{remk}

	Note that if $\rho \succeq n$, the $(v,\rho)$-local separation profile coincides with the usual separation profile for vertex-transitive graphs.

	The smallest (interesting) local profile is what we obtain choosing $\rho$ to be the generalised inverse of the volume growth: $\rho(n)=\gamma_v^{-1}(n) = \sup \{x \geq 0 \mid \gamma_v(x) \leq n\}$, with $\gamma_v(n)$ denoting the size of the ball of radius $n$ and centred at $v$. In that case we restrict the graphs investigated to lie in a ball of cardinality (almost) $n$.

	\begin{ques} Does there exists a vertex-transitive graph $G$ such that for some/any vertex  $v$ we have $\Sep_G^{\gamma_v^{-1},v} \nprec \Sep_G$  ?
		
		This question in linked with the issue of controlling the diameter of high separation graphs. Indeed, we expect those graphs to have a small diameter but finding such a counter-example would be very interesting.
	\end{ques}

	In the situations we investigate, 
	we get upper bounds in the case $\rho = \gamma^{-1}_v$. Then in what follows, we will restrict ourselves to this case. In this situation, the condition ``$|F|\leq n$'' is redundant and we will drop the $\rho$ from our notation:
	
	\[ \Sep_{G}^{v}(n) \vcentcolon = \Sep_G^{\gamma_v^{-1},v}(n) = 	\sup_{F < B_G(v,r) ; \left| B_G(v,r) \right| \leq n } |F| \cdot h(F) \]
	Therefore, we do not study this notion in its full generality, but we believe nevertheless that it can be relevant in some probabilistic contexts.
	
	Local separation profile will be studied in two cases: first, $\mathbb{Z}^d$ percolation clusters, then, graphs of isoperimetric dimension greater that one and of polynomial growth.
	
	\subsection{A local version of Theorem \ref{tpropnonsym-t}}

	\subsubsection{Statement of the Theorem}
	
	Before stating the theorem, we will introduce some notations for local isoperimetry.
	
	\begin{defn}
		We say that $F\subset G$ is $(n,v)$-optimal if:
		\begin{itemize}
			\item $F\subset B(v,n)$
			\item $\forall A\subset F\ \frac{|\del A|}{|A|} \geq \frac{|\del F|}{|F|}$
		\end{itemize}
		
		As before, we will say that an integer $r$ is $(n,v)$-optimal if there exists an $(n,v)$-optimal set of cardinality $r$.
		
	\end{defn}

	To adapt the previous result to our context, we need to introduce a local version of the isoperimetric profile:
	
	\begin{defn}
		Let $G$ be a graph. 
		Let $v \in G$ and $n$ be a positive integer. We define for any $r>0$: \[
		\Lambda_n^v(r)=\inf_{A\subset B_G(v,n), |A|\leq r}\frac{|\del A |}{|A|}
		\]
		
		This is a mixed profile between the classical and the isoperimetric profile inside the balls introduced by Tessera in \cite{{Tes11}} .
	\end{defn}
	
	We can now state the local versions of Lemma \ref{tbase-l} and Theorem \ref{tpropnonsym-t}. The proofs of the corresponding statements still work in this local context, we will not write them again:
	
	\begin{lem}\label{tbase-l-loc}
		Let $F$ be a $(n,v)$-optimal subset of a graph $G$. Then:
		\[2h(F)\geq \Lambda_n^v\left(\frac{|F|}{2}\right) - \Lambda_n^v(|F|)\]
	\end{lem}

	\begin{thm}\label{tpropnonsym-t-loc}
		Let $G$ be a connected infinite graph of bounded degree. Let $v\in G$, $n$ be a positive integer, and $ k $ be an integer.
		Assume there is a non-decreasing function $p:\left [0,k \right ] \to \left [0,|B(x,n)| \right ]$ so that for any $r \in \left [0,k \right ] $ there is an $r_{op} \in (r, p(r)]$ such that $r_{op}$ is optimal. 
		Choose $\eps \in (0,1) $.
		Let $r_1,r_2\in \left [0,k\right ]$ be such that $\Lambda_n^v(r_2) \leq (1-\eps)\Lambda_n^v(r_1) $.
		
		Then there exists an $r'\in [r_1,p(r_2)]$ such that 
		$$
		\boxed{
			\dfrac{\Sep^v(r')}{r'} \geq \eps\dfrac{\Lambda_n^v(r_1)}{4 \log(\frac{p(r_2)}{r_1}) + 4}.
		}
		$$
	\end{thm}

	\subsection{Application to polynomial graphs and $ \mathbb{Z}^{d} $ percolation clusters.}\label{sspoly1}
	
	We will apply Theorem \ref{thm-gen-loc} in graphs of polynomial growth and of dimension greater than one. We will call such a graph a \emph{polynomial graph}. We will show that around any point the separation is bounded below by a power of $ n $. We start with the definition of a polynomial graph:

		\begin{defn}\label{polygraph}
		Let $ G $ be a graph. Let $ d_1$ and $d_2 $ be two positive reals. We say that $ G $ is \emph{$ \left(d_1,d_2\right) $-polynomial} if there exist $ b,g>0 $ such that:
		
		\begin{itemize}
			\item For any vertex $ v $ and any integer $ n $ $ \left|B(v,n) \right| \leq b n^{d_2}$
			\item For any $ V \subset VG $, $|\partial V|\geq g|V|^{\frac{d_1-1}{d_1}}$
		\end{itemize}
		
	\end{defn}

	We were able to show a theorem that will apply both to polynomial graphs and to percolation clusters of $ \mathbb{Z}^{d} $. Therefore the assumptions of the theorem below are less restrictive, and polynomial graphs will be a particular case where they are satisfied. In particular, we do not require every subset of vertices to satisfy the isoperimetric inequality, but only if they are large enough.

	\begin{thm}\label{thm-gen-loc}
		Let $G$ be a connected infinite graph of bounded degree. 
		We assume that there exist $d_1, d_2>1$ and some functions $f, g, b>0$ such that for any vertex $ v $ and any integer $ n $:
		\begin{itemize}
			\item $|B(v,n)|\leq b(v)\cdot n^{d_2}$
			\item For any $A\subset B(v,n)$ such that $|A|\geq f(v,n)$, $|\partial A| \geq g(v)\cdot |A|^{1-1/d_1}$
		\end{itemize}
		
		We make the additional assumption that for any vertex $ v $ there exists an integer $ n_\omega $ such that for any integer $ n\geq n_{\omega} $ we have $ f(v,n) \leq \left | B(v,n) \right | $.
		
		Then for any $ \eta \in \left ( 0,1\right ) $, there exist $ c(v),K(v),\beta>0 $ (with $ \beta > d_1 $), such that for any vertex $ v $ and any large enough integer $ n $, when $ f(v,n) \leq c g(v)^{\frac{\beta - d_1}{d_1 - 1}} n^{\frac{d_1^2 (1 - \eta)^2}{d_2^2}}  $, then we have:
		
		\[ \Sep_{G}^{v}(n) \geq K g(v)^{\beta} n^{\alpha( 1 - \eta)}\text{, with $ \alpha = {\frac{d_1^2(d_1 - 1)}{d_2^3}} $.} \]
		
		Moreover, if $ d_1 = d_2 $ the conclusion is also true with $ \eta = 0 $. In this case, the constant $ K $ depends on the logarithm of $ g $.

	\end{thm}

Before proving this Theorem in subsection \ref{pf_thm_gen_loc}, we will state the corollaries we obtain in the two particular cases that interests us. First, to polynomial graphs:

\begin{thm}\label{thm_poly1}
	Let $ G $ be a $ \left(d_1,d_2\right) $-polynomial graph. Then for any $ \eta\in \left(0,1\right) $ there exists $ c>0 $ such that for any vertex $ v $ and any integer $ n $:
	
	\[ \Sep^v(n) \geq c n^{( 1 - \eta){\frac{d_1^2(d_1 - 1)}{d_2^3}}} \]
	
\end{thm}

\begin{remk}
	In the case where $d_1$ equals $d_2$ we get the expected exponent $\frac{d_1-1}{d_1}$, optimal in the case of vertex-transitive graphs, see \cite{BST}.
\end{remk}

As a second application, we study local separation in $\mathbb{Z}^{d}$ percolation clusters. We obtain the following theorem:

\begin{thm}\label{thm_perc}
	Let $p>p_c\left(\Z^d\right)$. Let $\omega$ be a percolation configuration of $\Z^d$ of parameter $p$. Let $\mathcal{C}_\infty$ be an (almost surely unique) infinite connected component of $\omega$. Let $ \varepsilon\in \left(0,1\right)$. Then there exist $c(d,p)>0$ and, for almost every $\omega$, an integer $l_\omega$ such that for any $n\geq l_\omega$ and for any $x\in\mathcal{C}_\infty$ such that $\|x\|_\infty\leq \exp\left(n^{(1 - \varepsilon)\frac{d}{d - 1}}\right)$, we have:
	\[\Sep_{\mathcal{C}_\infty}^{x}(n) \geq c n^{\frac{d-1}{d}}\]
\end{thm}

	This theorem will be deduced from a result on isoperimetry by G.~Pete, see \cite{Pete}. Note that this result had some beginnings in \cite{Bar} by Barlow and in \cite{BM} by Benjamini \& Mossel (see \cite{Pete} for details about the history of this result).

	\begin{thm*}[Pete \cite{Pete}, Corollary 1.3.] \label{thm_pete}
		For all $p > p_c(Z^d)$ there exist $c_3(d,p) > 0$, $\alpha(d,p)>0$ and (for almost all percolation configurations $\omega$) an integer $n_\omega$ such that for all $n > n_\omega$, all connected subsets $S \subset \mathcal{C}_\infty \cap [-n, n]^d$ with size $|S| \geq c_3 (\log n)^{\frac{d-1}{d}}$, we have $|\del_ {C_\infty} S|\geq \alpha |S|^{1-1/d}$.
	\end{thm*}

	\begin{proof}[Proof of Theorem \ref{thm_perc}]
	From this theorem, one can deduce that we can apply Theorem \ref{thm-gen-loc} to almost every percolation configuration with $ d_1 = d_2 = d $, $ f(v,n) = c_3 (\log \left(\|v\|_{\infty} + n\right))^{\frac{d-1}{d}}$, and $ g(v) = \alpha $.
	\end{proof}

	\subsection{Proof of Theorem \ref{thm-gen-loc}}\label{pf_thm_gen_loc}
	To show theorem \ref{thm-gen-loc}, we start with two lemmas. First, we can deduce from isoperimetry a lower bound on the growth of the graph:

	\begin{lem}\label{grwth_from_isop}
		Let $G$ be a connected infinite graph of bounded degree satisfying the assumptions of Theorem \ref{thm-gen-loc}. Let $v\in G$.
		Then there exists $ {b'}(v)>0 $ such that for any $ n $ at least equal to $n_{\omega}$ we have:
		\[ \left | B(v,n)\right | \geq {b'}(v) \cdot g(v)^{d_1}\cdot n^{d_1}  \]
	\end{lem}
	
	\begin{proof}
		We can substitute $n\mapsto |B(v,n)|$ with an piecewise affine function $B(t)$ that takes the same values on integer points. Then, for every $n> n_\omega$, we get:
		\begin{align*}
		B(n)^{1/d_1} - B(n_\omega)^{1/d_1}
		%%%%%%%%%%%%%%%%
		&= \frac{1}{d_1}\int_{n_\omega}^{n} \frac{{b'}(t)}{B(t)^{1-1/d_1}} dt \\
		%%%%%%%%%%%%%%%%
		&\geq \frac{1}{d_1}\sum_{n=n_\omega}^{n-1} \frac{B(r+1)-B(r)}{B(r+1)^{1-1/d_1}}\\
		%%%%%%%%%%%%%%%%
		&= \frac{1}{d_1}\sum_{n=n_\omega}^{n-1} \frac{B(r+1)-B(r)}{B(r)^{1-1/d_1}}\left(\frac{B(r)}{B(r+1)}\right)^{1-1/d_1}\\
		%%%%%%%%%%%%%%%%
		&\geq \frac{1}{d_1}\sum_{n=n_\omega}^{n-1} \frac{B(r+1)-B(r)}{B(r)^{1-1/d_1}}\frac{1}{{D}^{1-1/d_1}}\\
		%%%%%%%%%%%%%%%%
		&\geq \frac{1}{d_1}\sum_{n=n_\omega}^{n-1} \frac{|\partial B(v,r)|}{B(r)^{1-1/d_1}}\frac{1}{{D}^{2-1/d_1}}\\
		%%%%%%%%%%%%%%%%
		&\geq \frac{g(v)}{d_1{D}^{2-1/d_1}} \cdot \left( n - n_\omega\right)
		\end{align*}
		
		where $ D $ is a bound on the degrees of the vertices of $ G $.
\end{proof}
	
	Second, we can deduce an upper bound on isoperimetry of balls using growth:

	\begin{lem}\label{isop_from_grwth}
		Let $G$ be a connected infinite graph of bounded degree satisfying the assumptions of Theorem \ref{thm-gen-loc}. Let $v\in G$ and $\eta \in \left ( 0,1\right )$.
		
		Then there exists $a>0$ such that for any integer $n$ at least equal to $n_\omega^{\frac{1}{1-\eta}}$ there exists an integer $r$ between $n^{1-\eta}$ and $2n$ such that:
		
		\[
		\frac{|\partial B(v,r)|}{|B(v,r)|} \leq \frac{a}{|B(v,r)|^{1/d_1}}
		\]
		
		Moreover, if $ d_1 = d_2 $ the conclusion is also true for $ \eta = 0 $.
		
	\end{lem}
	
	To show this lemma, we will use the following facts, that we will prove later:

	\begin{fac}\label{doubling_card}
		Let $G$ be a connected infinite graph of bounded degree satisfying the assumptions of Theorem \ref{thm-gen-loc}. Let $v\in G$ and $\eta \in \left ( 0,1\right )$. 
		
		Then there exists $A>0$ such that for any non-negative integer $n$ there exists $m \in \left [ n^{1-\eta},n \right ] $ such that $|B(v,2m)|\leq A |B(v,m)|$.
		
		Moreover, if $d_1 = d_2$ the conclusion is also true for $\eta = 0$: there exists $A>0$ such that for any non-negative integer $n$ we have $|B(v,2n)| \leq A |B(v,n)|$.
	\end{fac}

	\begin{fac}\label{doubling_isop}
		Let $G$ be a connected infinite graph of bounded degree satisfying the assumptions of Theorem \ref{thm-gen-loc}.
		
		Let $ A>0 $ and $ v $ be a vertex of $ G $ and $m$ be an integer such that $|B(v,2m)|\leq A |B(v,m)|$. Then there exists an integer $r$ between $m$ and $2m$ such that:
		\[\frac{|\partial B(v,r)|}{|B(v,r)|} \leq \frac{\log(A)}{r}\]
		%With $C = \log(A)$
	\end{fac}
	
	Before proving those facts, we give a proof of Lemma \ref{isop_from_grwth}:
	
	\begin{proof}[Proof of Lemma \ref{isop_from_grwth}]
		According to the Facts \ref{doubling_card} and \ref{doubling_isop}, there exists $ A>0 $ such that for any non-negative integer $ n $ there exists $ r\in \left [n^{1-\eta},2n\right ] $ such that 	\[\frac{|\partial B(v,r)|}{|B(v,r)|} \leq \frac{\log(A)}{r}\]
		
		According to Lemma \ref{grwth_from_isop}, we have $ r\leq \dfrac{{\left |B(v,r)\right |}^{1/d_1}}{{b'}(v)^{1/d_1}g(v)} $. Therefore $\dfrac{|\partial B(v,r)|}{|B(v,r)|} \leq \dfrac{a}{|B(v,r)|^{1/d_1}} $ with $ a = g(v) \log(A) {b'}(v)^{1/d_1} $.
	\end{proof}
	
	We will now prove the Facts.
	
	\begin{proof}[Proof of Fact \ref{doubling_card}]
		Let $A$ be such that $\frac{\eta}{2}\log(A) \geq d_2 + \log(b+1) $, and let $n$ be a positive integer. Then:
		\begin{itemize}
			\item if $n\leq 1$ or $n\leq \exp\left (\dfrac{2}{\eta }\right )$, then up to taking a larger $A$, we can show that the conclusion of the lemma holds, since is $ G $ is of bounded degree.
			
			\item otherwise, we assume by contradiction that for any integer $m$ in the interval $\left [ n^{1-\eta},n\right ]$, we have $|B(x,2m)|> A \times |B(x,m)|$. Then we have:
			\begin{align*}
			|B(x,n)| &\geq A^{\log(n^{\eta}) - 1} |B(x,n^{1-\eta})|\\
			&\geq A^{\log(n^{\eta}) - 1}\\
			&\geq A^{\log(n^{\eta})/2}\qquad \text{as } n\geq \exp\left (\frac{2}{\eta}\right ) \\
			&\geq \exp \left ({\frac{\eta}{2} \log(n) \log(A)}\right )\\
			&\geq \exp(d_2 \log(n) + \log(b+1))\\
			&= (b+1) n^{d_2}
			\end{align*}
			(our logarithms and exponentials are in base $2$)
			
			This contradicts the assumption on the growth of the graph. 
		\end{itemize}
		
		If $ d_1 = d_2 $, the assumption on the growth of $ G $ and the conclusion of Lemma \ref{grwth_from_isop} give the announced result with $A=\frac{b}{{b'}}2^{d_2}$.
	\end{proof}
	
	\begin{proof}[Proof of Fact \ref{doubling_isop}]
		We assume by contradiction that for any $r$ between $n$ and $2n$ we have $\frac{|\partial B(v,r)|}{|B(v,r)|} > \frac{\log(A)}{r}$. That implies in particular the following inequality: $\frac{|B(v,r+1)| - |B(v,r)|}{|B(v,r)|} > \frac{\log(A)}{r}$.  Summing-up those inequalities, we have:
		
		\[\sum_{r=m}^{2m} \frac{|B(v,r+1)| - |B(v,r)|}{|B(v,r)|} > \log(A)\sum_{r=m}^{2m} \frac{1}{r}  \]
		
		Then we consider an piecewise affine function $B(t)$ that coincides with $|B(v,t)|$ on integer points. We get:

		\begin{align*}
		\log\left (\frac{B(2m)}{B(m)}\right ) 
		=  \int_{m}^{2m} \frac{{b'}(t)}{B(t)} dt 
		>
		\log\left (A\right )\int_{m}^{2m} \frac{1}{t} dt
		=\log\left (A\right )
		\end{align*}
		
		Therefore $B(2m) > A B(m)$, which is a contradiction.
	\end{proof}

	We are now able to prove Theorem \ref{thm-gen-loc}:
	
	\begin{proof}[Proof of Theorem \ref{thm-gen-loc}]
		Let $ v $ be a vertex of $ G $ and $ n $ be an integer at least equal to $ n_\omega $. We will require $ n $ to be (a priori) even larger in the following, satisfying some conditions that will appear later. Let $\eta$ be a real of the interval $\left(0,1\right)$, that may be equal to zero if $d_1=d_2$.
		
		According to the isoperimetric assumption, we have:
		
		\[\tag{is1}\label{isoplow} \forall r \in \left[ g(v)^{d_1} f(v,n)^{d_1}  , |B(v,n)| \right]\  \Lambda_n^v(r) \geq g(v) r^{-1/d_1} \]
		
		Indeed, let $ r $ be such an integer and let $ F $ a subset of $B(v,n) $ of cardinality at most $ r $. Two cases can occur:
		\begin{itemize}
			\item If $|F|\leq f(v,n)$, then since $G$ is infinite and connected, $|\partial F| \geq 1$. From the lower bound on $r$ we can deduce that $\dfrac{|\partial F |}{|F|} \geq \dfrac{1}{|F|} \geq \dfrac{1}{f(v,n)} \geq g(v) r^{-1/d_1} $
			\item Otherwise, we have by assumption $\dfrac{|\partial F |}{|F|} \geq g(v) {|F|}^{-1/d_1}\geq g(v) r^{-1/d_1}  $ 
		\end{itemize}
		
		%\max\left (6, \left ( 3 n_\omega \right )^{\frac{1}{1-\eta}}\right )
		Let $r$ be an integer in $\left[\max\left (4^{d_2}b(v), 4^{d_2} b(v) n_{\omega}^{\frac{d_2}{1-\eta}} \right ), |B(v,n)| \right] $. Let $r'$ be the biggest integer such that $\left |B\left (v,2r'\right )\right | \leq r $. According to Lemma \ref{isop_from_grwth}, there exists an integer $r''$ beetween $r'^{1-\eta}$ and $2r'$ such that $\dfrac{|\partial B(v,r'')|}{|B(v,r'')|} \leq \dfrac{a}{|B(v,r'')|^{1/d_1}}$. Since $B(v, 2r' + 2 ) \geq r$, we get from the growth assumption on $G$ that $r' \geq \dfrac{1}{2} {\left  (\dfrac{r}{b(v)} \right )}^{1/d_2} - 2 \geq \dfrac{1}{4} {\left  (\dfrac{r}{b(v)} \right )}^{1/d_2}$. Then we have $r'' \geq \dfrac{r^{(1-\eta)/d_2}}{4^{(1-\eta)}b(v)^{(1-\eta)/d_2}} \geq n_\omega$. Therefore we have: $\left | B(v,r'') \right | \geq {b'}(v) \cdot g(v)^{d_1}\cdot r''^{d_1} \geq  \dfrac{{b'}(v) g(v)^{d_1} }{4^{(1-\eta)d_1} b(v)^{(1-\eta)\frac{d_1}{d_2}}}r^{(1-\eta)\frac{d_1}{d_2}} $. We can deduce the following inequality, setting $ g'(v) = \dfrac{a. 4^{(1-\eta)} b(v)^{\frac{(1-\eta)}{d_2}}}{{b'}(v)^{1/d_1}} $:
		\[
		\tag{is2}\label{isopup}
		\forall r\in \left[\max\left (4^{d_2}b(v), 4^{d_2} b(v) n_{\omega}^{\frac{d_2}{1-\eta}} \right ), |B(v,n)| \right]\ 
		\Lambda_n^v(r) \leq \frac{g'(v)}{g(v)} r^{-\frac{(1-\eta)}{d_2}}
		\]		
		We omit for a while the issue of the length of the interval where our inequalities are verified. From the inequalities \eqref{isoplow} and  \eqref{isopup}, we can deduce that setting $s = {\left ( \dfrac{2g'(v)}{g(v)^2} \right )}^{\frac{d_2}{1 - \eta}}  $, if $r_2 \geq s \cdot r_1^{\frac{d_2}{d_1 (1 - \eta)}}  $, then $\Lambda(r_2)\leq \dfrac{1}{2} \Lambda(r_1)$. From this inequality we can deduce moreover that $p\colon r \mapsto s\cdot r^{\frac{d_2}{d_1(1-\eta)}}$ is a suitable function to apply Theorem \ref{tpropnonsym-t-loc}.
		
		Let $ r_1  $ be the biggest integer such that $ p(p(r_1)) \leq \left | B(v,n) \right | $. Then we have $ p(p(r_1+1)) \geq \left | B(v,n) \right | $. Since $ n $ is at most equal to $ n_\omega $, we can use Lemma \ref{grwth_from_isop}, which gives $ \left | B(v,n) \right | \geq {b'}\cdot g(v)^{d_1}\cdot n^{d_1} $. This yields to:
		$$ {\left(\frac{2g'(v)}{g(v)^2} \right)}^{\left(\frac{d_2}{1 - \eta} + \frac{d_2^2}{d_1(1 - \eta)^2}\right)}\cdot {\left(r_1 +1\right)}^{\frac{d_2^2}{d_1^2(1 - \eta)^2}} \geq {b'}\cdot g(v)^{d_1}\cdot n^{d_1} $$
		Therefore,
		\begin{align*}
		r_1 &\geq \frac{{b'}^{\frac{d_1^2 (1 - \eta)^2}{d_2^2}}}{{(2g'(v))}^{\frac{d_1^2 (1 - \eta)}{d_2} + \frac{1}{d_1}}} \cdot g(v)^{\left(\frac{d_1^3 (1 - \eta )^2}{d_2^2} + 2\frac{d_1^2 (1 - \eta)}{d_2}  + 2d_1 (1 - \eta) \right)}
		n^{\frac{d_1^3 (1 - \eta)^2}{d_2^2}} - 1 \\
		&\geq 	\frac{{b'}^{\frac{d_1^2 (1 - \eta)^2}{d_2^2}}}{{(3g'(v))}^{\frac{d_1^2 (1 - \eta)}{d_2} + \frac{1}{d_1}}}
		\cdot g(v)^{\left(\frac{d_1^3 (1 - \eta )^2}{d_2^2} + 2\frac{d_1^2 (1 - \eta)}{d_2}  + 2d_1 (1 - \eta) \right)}
		n^{\frac{d_1^3 (1 - \eta)^2}{d_2^2}}
		\text{, if $ n $ is large enough.}
		\end{align*}
		Then if $ n $ is large enough $ r_1  $ is in the validity domain of \eqref{isopup}. Moreover, if we set
		
		\begin{align*}
			c &= \frac{{b'}^{\frac{d_1 (1 - \eta)^2}{d_2^2}}}{{(3g'(v))}^{\frac{d_1 (1 - \eta)}{d_2} + \frac{1}{d_1^2}}} \\\\
			%\text{and } 
			\beta &= \frac{d_1^2 (d_1 - 1) (1 - \eta )^2}{d_2^2} + 2\frac{d_1 (d_1 - 1) (1 - \eta)}{d_2}  + 2(d_1 - 1) (1 - \eta) +1 
		\end{align*}
		we get that if $ f(v,n) \leq c g(v)^{\frac{\beta - d_1}{d_1 - 1}} n^{\frac{d_1^2 (1 - \eta)^2}{d_2^2}} $, then $ r_1 $ is in the validity domain of \eqref{isoplow}. We find out the condition on $ f(v,n) $ that is made in the statement of Theorem \ref{thm-gen-loc}. Under this assumption, we can apply Theorem \ref{tpropnonsym-t-loc} with $ r_2 = p(r_1) $. This gives:		
		\begin{align*}
		\Sep^v(\left|B(v,n)\right|) &\geq \Sep^v(p(r_2))\\
		&\geq r_1 \frac{\Lambda_n^v(r_1)}{8\log(\frac{p(r_2)}{r_1}) + 8}\\
		\end{align*}
		First, from \eqref{isoplow} and the lower bound on $ r_1 $, we have: $$r_1 \Lambda_n^v(r_1) \geq g(v)r_1^{\frac{d_1 - 1}{d_1}} \geq \frac{{b'}^{\frac{d_1 (d_1 - 1) (1 - \eta)^2}{d_2^2}}}{{(3g'(v))}^{\frac{d_1 (d_1 - 1) (1 - \eta)}{d_2} + \frac{d_1 - 1}{d_1^2}}}
		\cdot g(v)^{\beta}
		n^{\frac{d_1^2(d_1 - 1) (1 - \eta)^2}{d_2^2}}$$
		
		Second, since $p(r_2) \leq |B(v,n)| \leq b(v)\cdot n^{d_2} $, from the lower bound on $ r_1 $ we have: 
		\begin{align*}
		8\log(\frac{p(r_2)}{r_1}) + 8 &\leq 8\log\left(b(v)n^{d_2}\right) - 8\log\left(\frac{{b'}^{\frac{d_1^2 (1 - \eta)^2}{d_2^2}}}{{(3g'(v))}^{\frac{d_1^2 (1 - \eta)}{d_2} + \frac{1}{d_1}}}
		\cdot g(v)^{\left(\frac{d_1^3 (1 - \eta )^2}{d_2^2} + 2\frac{d_1^2 (1 - \eta)}{d_2}  + 2d_1 (1 - \eta) \right)}
		n^{\frac{d_1^3 (1 - \eta)^2}{d_2^2}}\right) + 8\\
		&=  8 \dfrac{d_2^3 - d_1^3 (1 - \eta)^2}{d_2^2} \log(n) + 8\log\left(\frac{{b(v)(3g'(v))}^{\frac{d_1^2 (1 - \eta)}{d_2} + \frac{1}{d_1}}}{{b'}^{\frac{d_1^2 (1 - \eta)^2}{d_2^2}}g(v)^{\left(\frac{d_1^3 (1 - \eta )^2}{d_2^2} + 2\frac{d_1^2 (1 - \eta)}{d_2}  + 2d_1 (1 - \eta) \right)}}
		\right) + 8	
		\end{align*}
		
		Finally:
		
		\begin{itemize}
			\item if $ d_1\neq d_2 $, we have for $ n $ large enough:
			\begin{align*}
			\Sep^v(\left | B(v,n) \right |) &\geq \frac{{{b'}^{\frac{d_1 (d_1 - 1) (1 - \eta)^2}{d_2^2}}}
			}{9 \frac{d_2^3 - d_1^3 (1 - \eta)^2}{d_2^2}{{(3g'(v))}^{\frac{d_1 (d_1 - 1) (1 - \eta)}{d_2} + \frac{d_1 - 1}{d_1^2}}}		} g(v)^{\beta} \dfrac{n^{{\frac{d_1^2(d_1 - 1) (1 - \eta)^{2}}{d_2}}}}{\log(n)} \\
			&\geq \frac{{{b'}^{\frac{d_1 (d_1 - 1) (1 - \eta)^2}{d_2^2}}}
			}{9 \frac{d_2^3 - d_1^3 (1 - \eta)^2}{d_2^2}{{(3g'(v))}^{\frac{d_1 (d_1 - 1) (1 - \eta)}{d_2} + \frac{d_1 - 1}{d_1^2}}}	{b^{{\frac{d_1^2(d_1 - 1) (1 - \eta)^{2}}{d_2^3}}}}	} g(v)^{\beta}	 \dfrac{{\left | B(v,n) \right |}^{{\frac{d_1^2(d_1 - 1) (1 - \eta)^{2}}{d_2^3}}}}{\log(\left | B(v,n) \right |)}
			\end{align*}

			Therefore we have:
			\[
			\Sep^v(N) \geq K g(v)^{\beta} \dfrac{{N}^{\alpha}}{\log(N)}
			\]

			if $ N $ is large enough, with: ($ D $ denotes a bound on the degrees of the vertices of $ G $)
			\begin{align*}
			&K = \frac{{{b'}^{\frac{d_1 (d_1 - 1) (1 - \eta)^2}{d_2^2}}}
			}{9 \frac{d_2^3 - d_1^3 (1 - \eta)^2}{d_2^2}{{(3g'(v))}^{\frac{d_1 (d_1 - 1) (1 - \eta)}{d_2} + \frac{d_1 - 1}{d_1^2}}}	{{\left(Db+b\right)}^{{\frac{d_1^2(d_1 - 1) (1 - \eta)^{2}}{d_2^3}}}}	}\\
			&\alpha = {\frac{d_1^2(d_1 - 1) (1 - \eta)^{2}}{d_2^3}} \\
			&\beta = \frac{d_1^2 (d_1 - 1) (1 - \eta )^2}{d_2^2} + 2\frac{d_1 (d_1 - 1) (1 - \eta)}{d_2}  + 2(d_1 - 1) (1 - \eta) +1				 
			\end{align*}
			
			Substituting $ (1 - \eta)^{2} $ with $ (1 - \eta) $ and removing the $ \log(n) $, taking a larger $\eta$, 
			we are done.
			
			\item if $d_1 = d_2$, we have for $ n $ large enough: ($ D $ denotes a bound on the degrees of the vertices of $ G $)
			\begin{align*}
			\Sep^v(\left | B(v,n) \right |) &\geq \left(Db+b\right)^{{\frac{d_1^2(d_1 - 1 ) (1 - \eta)^{2}}{d_2^3}}} K
			\cdot g(v)^{\beta}
			n^{\frac{d_1^2(d_1 - 1) (1 - \eta)^2}{d_2^2}} \\
			&\geq {\left(D+1\right)}^{{\frac{d_1^2(d_1 - 1 ) (1 - \eta)^{2}}{d_2^3}}} K
			\cdot g(v)^{\beta}
			{\left|B(v,n)\right|}^{\frac{d_1^2(d_1 - 1) (1 - \eta)^2}{d_2^2}}
			\end{align*}
			
			Therefore we have for any integer $ N $:
			\[
			\Sep^v(N) \geq K g(v)^{\beta} {N}^{\alpha}
			\]

if $ N $ is large enough, with:\\
$K =  \frac{{b'}^{\frac{d_1 (d_1 - 1) (1 - \eta)^2}{d_2^2}}}{{(3g'(v))}^{\frac{d_1 (d_1 - 1) (1 - \eta)}{d_2} + \frac{d_1 - 1}{d_1^2}}\left(Db+b\right)^{{\frac{(d_1^{3} - d_1^2) (1 - \eta)^{2}}{d_2^3}}}}\left( 8\log\left(\frac{{b(v)(3g'(v))}^{\frac{d_1^2 (1 - \eta)}{d_2} + \frac{1}{d_1}}}{{b'}^{\frac{d_1^2 (1 - \eta)^2}{d_2^2}}g(v)^{\left(\frac{d_1^3 (1 - \eta )^2}{d_2^2} + 2\frac{d_1^2 (1 - \eta)}{d_2}  + 2d_1 (1 - \eta) \right)}}\right) + 8	\right)^{-1}$ \\
$\alpha = {\frac{(d_1^{3} - d_1^2) (1 - \eta)^{2}}{d_2^3}}$\\
$\beta = \frac{- d_1^{2}(1-\eta)^2 - d_1 d_2 (1 - \eta) + d_1^3 (1 - \eta)^2 + d_1^2 d_2 (1 - \eta) + d_1 d_2^2 }{d_2^2}$\\
\text{ (note that in this case $ K $ depends on $ g $)} \qedhere
		\end{itemize}

	\end{proof}

	\subsection{Another approach for polynomial graphs.}\label{sspoly2}

	In this subsection, we study local separation in graphs of polynomial growth and of isoperimetric dimension greater than 1. Using a more abstract and simple approach, we show again that around any point the separation is bounded below by a power of $n$, that improves Theorem \ref{thm-gen-loc} in some cases. We will prove a statement in a slightly more general context than polynomial graphs, with a local flavour, which is very natural regarding to the proof. We will then formulate the theorem in the setting of polynomial graphs (Theorem \ref{thm_poly2}). Here is our theorem:

	\begin{thm}
		Let $G$ be an infinite graph of bounded degree such that there exists $d_2 \geq d_1 > 1$ and two positive functions $b(v)$ and $g(v,n)$ such that for any vertex $v$ and any positive integer $n$:
		\begin{itemize}
			\item $\gamma_v(n) \vcentcolon = |B(v,n)|\leq b(v) n^{d_2}$.
			\item For any $V \subset B\left (v,n\right )$, $|\partial V|\geq g(v,\gamma_v(n))|V|^{\frac{d_1-1}{d_1}}$.
		\end{itemize}
		
		We assume moreover that $d_1^2 > d_2 - d_1$. Then for any $\eta>0$ there exists $s>0$ depending only on $d_1$, $d_2$, $b$ and $\eta$ such that for any positive integer $n$ and any vertex $v$:
		\begin{align*}
				\Sep_G^v(n) \geq s \cdot g(v,n)^{\beta} \cdot n^{(1-\eta)\alpha}\quad \text{with } \alpha &= \frac{(d_1-1)(d_1^2 - (d_2 - d_1))}{d_1^2d_2}\\
		\text{ and }\beta &= {\frac{d_1^2 +d_1 - 1}{d_1}}
		\end{align*}
		Moreover, if $d_1=d_2$ the conclusion is also true for $\eta = 0$.

	\end{thm}

	\begin{remk}
		The conclusion of the theorem implies in particular that the classical (or \emph{global}) separation profile is bounded below:
		For any $\eta >0$ there exists $s(v,\eta)>0$ such that for any positive integer $n$:
		\[
		\Sep_G(n) \geq s \cdot g(v,n)^{\beta} \cdot n^{(1-\eta)\alpha}
		\]

	\end{remk}

	This theorem follows, using the terminology introduced in Definition \ref{polygraph}:
	
	\begin{thm}\label{thm_poly2}
		Let $ G $ be a $ \left(d_1,d_2\right) $-polynomial graph such that $ d_2 - d_1 < d_1^2 $. Then for any $ \eta\in \left(0,1\right) $ there exists $ c>0 $ such that for any vertex $ v $ and any integer $ n $:
		
		\[ \Sep^v_G(n) \geq c n^{( 1 - \eta){\alpha}}\quad \text{with }\alpha = \frac{(d_1-1)(d_1^2 - (d_2 - d_1))}{d_1^2d_2} \]
		Moreover, if $d_1=d_2$ the conclusion is also true for $\eta = 0$.
	\end{thm}
	\begin{remk}
		As in Theorem \ref{thm_poly1}, in the case where $d_1$ equals $d_2$ we get the expected exponent $\frac{d_1-1}{d_1}$, optimal in the case of vertex-transitive graphs. If $ d_1 $ is smaller than $ d_2 $ one can notice that Theorems \ref{thm_poly1} and \ref{thm_poly2} do not give the same exponents (the best can be given by one or the other, depending on the values of $ d_1 $ and $ d_2 $), which is an interesting demonstration of the fact that, despite the use of the same ingredients, the two approaches are essentially different.
	\end{remk}

	\begin{proof} Let us explain the strategy of this proof. 
		We will call \emph{isoperimetric ratio} of a set the ratio between the size of its boundary and its size, $\frac{|\partial \cdot|}{|\cdot|}$. 
		Our a goal is to find, for any $n$, a subset $X$ of $B(v,n)$ for which we can bound below its cardinality and its Cheeger constant in order to get a bound on $|X|h(X)$.
		Adapting slightly the proof to our lemma, we see that to bound its Cheeger constant, it suffices for $X$ to verify two conditions: first that it has a lower (or equal) isoperimetric ratio than its subsets, and second that the isoperimetric ratio of its small (less than a half) subsets is bigger, by a controlled factor greater than 1. 
		To get those properties, we proceed recursively: starting from a ball $B(v,n)$ , we take smaller and smaller subsets that violates the second condition, and when there is no such small subset, we finally take a subset of the resulting set that minimises the isoperimetric ratio. 
		Our hypothesis on the growth of the graph gives an upper bound on the isoperimetric ratio  the size of the boundary of $B(v,n)$, and the hypothesis on the isoperimetric dimension ensures a lower bound on the cardinality of the final set and on its isoperimetric ratio, leading to a bound on its Cheeger constant.
		
		Let $v$ be a vertex of $G$. We start with a doubling property of the graph $G$:
		
		Let $\eta$ be a real of the interval $\left(0,1\right)$, that may be equal to zero if $d_1=d_2$. Let $n$ be an integer at least equal to $2$. Let $m$, $A$, and $r$ be given by Facts \ref{doubling_card} and \ref{doubling_isop}. Then we have:
		\begin{itemize}
			\item $n^{1-\eta} \leq r \leq 2n $
			\item $\frac{|\partial B(v,r)|}{|B(v,r)|} \leq \frac{\log(A)}{r}$
		\end{itemize}
		
		Let's write $F_1=B(v,r)$ the ball of $G$ centred at $v$ of radius $r$. Let $g = g\left(v,\left|B(v,2n)\right|\right)$ and $\eps$ be a positive real small enough so that $2^{\frac{1}{d_1+1}}\leq \frac{2^{1/d_1}}{1+\eps}$.
		
		We construct a finite decreasing sequence $\left (F_i\right )_{i}$ by induction, in the following way: let $i$ be a positive integer. If $F_i$ is defined, then:
		\begin{itemize}
			\item If there exists a subset of $A$ of $F_i$ such that $|A|\leq \frac{|F_i|}{2}$ and $\frac{|\partial A|}{|A|}\leq \left (1+\eps\right ) \frac{|\partial F_i|}{|F_i|}$, then we take $F_{i+1}$ being such a set.
			\item Otherwise, we stop the sequence.
		\end{itemize}
		
		Let $k$ denote the number of terms of this sequence. From the isoperimetric dimension hypothesis we have: $|F_k|^{-1/d_1}\leq \frac{1}{g}\frac{|\partial F_k|}{|F_k|}
		\leq \frac{\left (1+\eps\right )^k}{g} \frac{|\partial F_1|}{|F_1|}
		\leq \frac{\left (1+\eps\right )^k}{g} \frac{\log(A)}{r}
		$, therefore we can deduce that $|F_k|\geq \frac{g^{d_1}r^{d_1}}{\log(A)^{d_1}\left (1+\eps \right )^{k d_1}}$.
		
		By construction, we have $|F_k|\leq 2^{-k} |F_1|$. Hence, we can deduce that \[
		2^{k/d_1} |F_1|^{-1/d_1}\leq |F_k|^{-1/d_1} \leq \frac{\left (1+\eps\right )^k}{g} \dfrac{\log(A)}{r}\] 
		which means that 
		$2^{\frac{k}{d_1+1}}
		\leq \left (\frac{2^{1/d_1}}{1+\eps} \right )^k 
		\leq \frac{\log(A)}{g}  \dfrac{|F_1|^{1/d_1}}{r}
		\leq \frac{\log(A)b^{1/d_1}}{g}  \dfrac{r^{d_2/d_1}}{r}
		= \frac{\log(A)b^{1/d_1}}{g} r^{\frac{d_2-d_1}{d_1}}
		$.

		Then, since $\left (1+\eps \right )^{k d_1} \leq 2^{\left (\frac{1}{d_1}-\frac{1}{d_1+1}\right )kd_1} =  2^{\frac{k}{d_1+1}}$, we can deduce that, with $c=\log(A)^{-(d_1 + 1)} b^{-1/d_1}$, 
		\[|F_k|\geq \displaystyle c\cdot  g^{d_1+1} \cdot \frac{r^{d_1}}{r^{\frac{d_2-d_1}{d_1}}} = c \cdot  g^{d_1+1} \cdot r^{\frac{d_1^2 - (d_2 - d_1)}{d_1}}\]
		We can take a final set $X$ minimising $\frac{|\partial \cdot|}{|\cdot|}$ among subsets of $F_k$.
		Therefore, $X$ satisfies the following properties :
		\begin{itemize} 
			\item $\forall Y\subset X\ \frac{|\partial Y|}{|Y|}\geq \frac{|\partial X|}{|X|}$
			\item $\forall Y\subset X\ \left ( |Y|\leq\frac{|X|}{2} \Rightarrow \frac{|\partial Y|}{|Y|}\geq \left ( 1+\eps \right )\frac{|\partial X|}{|X|} \right )$
		\end{itemize}
		
		\begin{lem}
			From those two properties, we can deduce the following inequality:
			\[
			2 h(X)\geq \eps\frac{|\partial X|}{|X|}
			\]
		\end{lem}
		
		\begin{proof}
			The proof is very similar to the proof of Lemma \ref{tbase-l}.
			Let $F_1$ be a subset of $X$ such that $|F_1|\leq \frac{|X|}{2}$. We denote $F_2 = X\setminus F_1$. Then we have:
			
			\begin{align*}
			2|\partial_X F_1| &= |\partial F_1| + |\partial F_2| - |\partial X| \\
			&\geq \left (1+\eps \right )\dfrac{|\partial X|}{|X|}|F_1| +  \dfrac{|\partial X|}{|X|}|F_2| - |\partial X|\\
			&= \eps \dfrac{|\partial X|}{|X|}|F_1| + \dfrac{|\partial X|}{|X|}\left (|F_1| + |F_2|\right ) - |\partial X|\\
			&= \eps \dfrac{|\partial X|}{|X|}|F_1|
			\end{align*}
			
			Then we have $2\frac{|\partial_X F_1|}{|F_1|} \geq \eps \frac{|\partial X|}{|X|} $. Since this is true for any subset $F_1$ of $X$ containing at most half of its points, we have shown the announced inequality.
		\end{proof}

		By construction of $F_k$, we have $|X|\geq |F_k|/2$. We have:
		\begin{align*}
		|X|h(X) &\geq \frac{\eps}{2} \cdot|\partial X|\\
		&\geq \frac{\eps}{2} g \cdot |X|^{\frac{d_1-1}{d_1}}\\
		&\geq \frac{\eps}{2} g 2^{\frac{1-d_1}{d_1}} \cdot|F_k|^{\frac{d_1-1}{d_1}}\\
		&\geq  \frac{\eps}{2}  g  2^{\frac{1-d_1}{d_1}}  c^{\frac{d_1-1}{d_1}} g^{\frac{(d_1-1)(d_1+1)}{d_1}} \cdot r^{\frac{(d_1-1)(d_1^2 - (d_2 - d_1))}{d_1^2}}\\
		&\geq c' \cdot g^{\frac{d_1^2 +d_1 -1}{d_1}}\cdot n^{d_2(1-\eta)\alpha}\\
		\end{align*}
		With $c' = \frac{\eps}{2}\cdot 2^{\frac{1-d_1}{d_1}} \cdot c^{\frac{d_1-1}{d_1}} $ and $\alpha = \frac{(d_1-1)(d_1^2 - (d_2 - d_1))}{d_1^2d_2}$.

We have shown that there exists a positive constant $c'$ such that for any integer $n\geq 2$ and any vertex $v$, we have: 
\[
\begin{array}{rl}
\displaystyle \Sep_G^v(|B(v,2n)|) 
 &\geq \displaystyle c' {g\left (v,|B(v,2n)|\right )}^{\frac{d_1^2 +d_1 - 1}{d_1}}  \cdot n^{d_2(1-\eta)\alpha} \\
&\geq \displaystyle \frac{c'}{b^{(1-\eta)\alpha} 2^{d_2 {(1-\eta)\alpha}}} {g\left (v,|B(v,2n)|\right )}^{\frac{d_1^2 +d_1 - 1}{d_1}}  \cdot {\left|B(v,2n)\right|}^{(1-\eta)\alpha} 
\end{array}
\]
		
		The announced result follows.
	\end{proof}

\textsc{Laboratoire de Mathématiques d’Orsay, Université Paris-Sud, CNRS, Université Paris-Saclay, 91405 Orsay, France}\\
\textit{E-mail address:} \url{corentin.le-coz@u-psud.fr}
	
	%________________________________________________________

\textsc{Institut für Geometrie, TU Dresden, Zellescher Weg 12-14, 01069 Dresden, \linebreak Germany}\\ 
\textit{E-mail address:} \url{antoine.gournay@tu-dresden.de}

	%________________________________________________________

\end{document}